\documentclass[a4paper, 12pt]{amsart}

\usepackage[colorlinks,citecolor=marron,linkcolor=violet]{hyperref}
\usepackage{xcolor} 
\definecolor{marron}{HTML}{BA2908}
\definecolor{violet}{HTML}{2900BA}

\usepackage[lmargin=2.5 cm,rmargin=2.5 cm,tmargin=3.5cm,bmargin=2.5cm,paper=a4paper]{geometry}

\usepackage{color}
\definecolor{mg}{rgb}   {0.85,  0.,    0.85}

\usepackage{amssymb}
\usepackage{graphicx}
\usepackage{times,mathrsfs}
\usepackage{rsfso}
\usepackage[english]{babel}
\usepackage{enumerate}
\usepackage{enumitem}
\usepackage{csquotes}

\usepackage{caption}
\usepackage{subfig}

\makeatletter
	
	\@addtoreset{equation}{section}
\makeatother

\newtheorem{thm}{Theorem}[section]
\newtheorem{lem}[thm]{Lemma}
\newtheorem{prop}[thm]{Proposition}
\newtheorem{cor}[thm]{Corollary}

\theoremstyle{definition}

\newtheorem{nota}[thm]{Notation}
\newtheorem{ass}[thm]{Assumption}

\theoremstyle{remark}

\renewcommand{\Re}{\operatorname{\textrm{\textup{Re}}}}

\newcommand{\IMS}{\mathsf{loc}}
\newcommand{\Hess}{\mathop{\sf{Hess}}}

\newcommand{\cB}{\mathcal{B}}

\newcommand{\Cr}{\mathcal{X}_{\varepsilon}}
\newcommand{\N}{\mathbb{N}}

\newcommand{\R}{\mathbb{R}}
\newcommand{\rd}{\mathrm{d}}
\newcommand{\re}{\mathrm{e}}
\newcommand{\rS}{\mathrm{S}}

\newcommand{\Dom}{\mathop{\sf Dom}}
\newcommand{\loc}{^{\sf loc}}

\newcommand{\dist}{\rd}
\newcommand{\Spec}{\mathfrak{S}}
\newcommand{\supp}{\operatorname{\textrm{supp}}}
\newcommand{\spa}{\operatorname{\textrm{span}}}

\newcommand{\BB}{\mathbf{B}}
\newcommand{\AAA}{\mathbf{A}}
\newcommand{\sta}{^{\!\textstyle\star}}
\newcommand{\sts}{^{\!\scriptstyle\star}}
\newcommand{\Astar}{\mathbf{A}^{\!\textstyle\star}}
\newcommand{\cP}{\mathcal{P}}

\newcommand{\Fq}{\mathcal{Q}_{h,\Omega}^{\AAA}}
\newcommand{\Fqzz}{\mathcal{Q}_{h}^{\AAA^{\cx_0}}}
\newcommand{\Fqstar}{\mathcal{Q}_{h,\Omega}^{\AAA^{\!\scriptstyle\star}}}

\newcommand{\sA}{\mathsf{A}}
\newcommand{\sN}{\mathsf{N}}
\newcommand{\sP}{\mathsf{P}}
\newcommand{\sQ}{\mathsf{Q}}
\newcommand{\sV}{\mathsf{V}}
\newcommand{\sX}{\mathsf{X}}

\newcommand{\cg}{\mathsf{g}}
\newcommand{\cx}{\mathsf{x}}
\newcommand{\cy}{\mathsf{y}}
\newcommand{\cY}{\mathsf{Y}}

\newcommand{\cL}{\mathcal{L}}

\newcommand{\fs}{\mathfrak{s}}
\newcommand{\ft}{\mathfrak{t}}
\newcommand{\fL}{\mathfrak{L}}

\newcommand{\CBSigma}{\Lambda_1^{\BB}}

\newcommand{\Omegazz}{\cB(\cx_0,r_0)}

\newcommand{\Op}{\mathcal{P}_{h,\Omega}^{\AAA}}
\newcommand{\Opzz}{\mathcal{P}_{h}^{\AAA^{\cx_0}}}
\newcommand{\Oploc}{\mathcal{P}_{h,\Omegazz}^{\AAA}}

\newcommand{\Cims}{K_\mathsf{loc}}

\newcommand{\Dir}{\mathsf{Dir}}

\newcommand{\chifrakUU}{\mathfrak{X}^{[1]}}

\title[On the semiclassical magnetic Laplacian]{On the semiclassical Laplacian with magnetic field\newline  having self-intersecting zero set}

\author[M. Dauge]{Monique Dauge} 
\address{\textbf{M. Dauge}, Univ. Rennes, CNRS, IRMAR - UMR 6625, F-35000 Rennes, France}
\email{\href{mailto:monique.dauge@univ-rennes1.fr}{monique.dauge@univ-rennes1.fr}}
\urladdr{\href{https://perso.univ-rennes1.fr/monique.dauge/}{https://perso.univ-rennes1.fr/monique.dauge/}}

\author[J-P. Miqueu]{Jean-Philippe Miqueu} 
\address{\textbf{J-P. Miqueu}}
\email{\href{mailto:jean-philippe.miqueu@laposte.net}{jean-philippe.miqueu@laposte.net}}
\urladdr{\href{http://www.jean-philippe-miqueu.com}{http://www.jean-philippe-miqueu.com}}

\author[N. Raymond]{Nicolas Raymond}
\address{\textbf{N. Raymond}}
\email{\href{mailto:nicolas.raymond@univ-rennes1.fr}{nicolas.raymond@univ-rennes1.fr}}
\urladdr{\href{http://nraymond.perso.math.cnrs.fr}{http://nraymond.perso.math.cnrs.fr}}

\begin{document}



\begin{abstract}
This paper is devoted to the spectral analysis of the Neumann realization of the 2D magnetic Laplacian with semiclassical parameter $h>0$ in the case when the magnetic field vanishes along a smooth curve which crosses itself inside a bounded domain.  We investigate the behavior of its eigenpairs in the limit $h\to0$. We show that each crossing point acts as a potential well, generating a new decay scale of $h^{3/2}$ for the lowest eigenvalues, as well as exponential concentration for eigenvectors around the set of crossing points. These properties are consequences of the nature of associated model problems in $\R^2$ for which the zero set of the magnetic field is the union of two straight lines. In this paper we also analyze the spectrum of model problems when the angle between the two straight lines tends to $0$.
\end{abstract}

\maketitle

\section{Introduction}
\subsection{The magnetic Laplacian}
Let $\Omega$ be a bounded, smooth and simply connected open set of $\R^2$, and $\AAA\in \mathcal{C}^{\infty}(\overline{\Omega},\R^2)$ be a regular potential vector. For $h>0$, we consider the self-adjoint operator
\begin{equation*}
   \Op=(-ih\nabla+\AAA)^2,
\end{equation*}
with domain
\[\Dom(\Op)=\{u \in H^2(\Omega), (-ih\nabla+\AAA)u \cdot \nu  =0 \quad \text{on} \quad \partial\Omega \},\]
where $\nu$ is the outward pointing normal at the boundary of $\Omega$. 

The operator $\Op$ has compact resolvent and is associated with the quadratic form $\Fq$ defined on the form domain $\Dom(\Fq)=H^1(\Omega)$
by
\begin{equation}\label{QF}
\displaystyle{\Fq(u)= \int_{\Omega}|(-ih\nabla+\AAA)u(\cx)|^2\, \mathrm{d}\cx}\,.
\end{equation}
Here $\cx=(x_1,x_2)$ denotes Cartesian coordinates in $\R^2$.
We have the gauge invariance
\begin{equation}\label{gauge}
\re^{-i\phi/h}\left(-ih\nabla+\AAA\right)^2 \re^{i\phi/h}=\left(-ih\nabla+\AAA+\nabla\phi\right)^2\,,
\end{equation} 
for any $\phi\in H^2(\Omega)$. Therefore, the spectrum of $\Op$ only depends on the magnetic field 
\begin{equation}\label{eq:rot}
   \BB=\nabla\times\AAA=\partial_{2}A_{1}-\partial_{1}A_{2}\,.
\end{equation}

\begin{nota}\label{nota:sp}
We denote by $\lambda_{n}(h)$ the $n$-th eigenvalue (with multiplicity) of $\Op$. In all of the paper, $\Spec(\mathcal{P})$ denotes the spectrum of any operator $\mathcal{P}$.
\end{nota}

We are interested in the behavior of the eigenvalues $\lambda_{n}(h)$ and their associated eigenfunctions in the semiclassical limit $h\to0$ for special configurations of the magnetic field $\BB$.

\subsection{Motivations and context}
The spectral analysis of the magnetic Laplacian comes from the theory of superconductivity in which the magnetic Laplacian appears in the study of the third critical field of the Ginzburg-Landau functional (see for instance \cite{SJ70} and also the books \cite{FH} and \cite{R16}, and the references therein). The regime when $h$ goes to $0$ (called the semiclassical limit) is equivalent to the strong magnetic field limit which is often involved in applications. In this paper, we restrict to dimension two.

\subsubsection{Overview of the literature}
In the past two decades, most of the contributions dealt with non-vanishing magnetic fields. We can refer for instance to the works by Bolley \& Helffer \cite{BH97}, Bauman, Phillips \& Tang \cite{BPT98}, 
del Pino, Felmer \& Sternberg \cite{dPFPS00}, Helffer \& Morame \cite{HM01}, Bonnaillie-No\"el \cite{BN05}, Lu \& Pan \cite{LP99}, Raymond \cite{R09}, Bonnaillie-No\"el \& Dauge \cite{BD06},
Bonnaillie-No\"el \& Fournais \cite{BF07}, Raymond \& Vu-Ngoc \cite{RV13}.

The present paper is devoted to the case of vanishing magnetic fields. Such an investigation was initially motivated by a paper of Montgomery \cite{Mont}, followed by the contributions of Helffer \& Morame \cite{HM96}, Helffer \& Kordyukov \cite{HK09, H09}, and Dombrowski \& Raymond \cite{RD}. The aforementioned  papers do not investigate the case when the zero set of the magnetic field $\BB$ intersects the boundary. This was the purpose of the work by Pan \& Kwek \cite{PK} and \cite{artJP}. We can find in \cite{PK} a one term asymptotics of the first eigenvalue $\lambda_1(h)$. The paper \cite{artJP} 
establishes a sharper result by giving an explicit control of the remainder as well as expansions of all the eigenvalues and eigenfunctions, as the semiclassical parameter $h$ goes to $0$, under suitable assumptions when the zero set of $\BB$ does not self-intersect. 

\subsubsection{When the zero set of $\BB$ self-intersects} 
In the present paper, we want to include non-degenerate quadratic cancellations inside the domain, which is a new configuration in the investigations about vanishing magnetic fields.
\begin{ass}\label{hypcadre2}
Let 
\[\Gamma=\{\cx\in\overline{\Omega} : \BB(\cx)=0\}\,,\]
and assume that $\Gamma\neq\emptyset$.
We work under the following assumptions
\begin{enumerate}[label=\roman*)]
\item\label{ii} 
The set $$\Sigma=\{\cx\in\Gamma, \nabla\BB(\cx)=0\}$$ is non-empty, finite and such that  $\partial\Omega\cap\Sigma=\emptyset$.
\item \label{iii} 
For any $\cx\in \Sigma$, the Hessian matrix $\Hess\BB(\cx)$ of the magnetic field at the point $\cx$ has two non-zero eigenvalues with opposite signs.
\item \label{i}The set $\Gamma\cap\partial\Omega$ is finite, and in each of these intersection points, $\Gamma$ is non tangent to $\partial\Omega$.
\end{enumerate}
\end{ass}

Note that assumptions \ref{ii}-\ref{iii} imply that the set $\Gamma$ is a simple curve in a neighborhood of each of its intersection points with $\partial\Omega$. Moreover, the set $\Sigma$ is made of isolated points which are locally the intersection point of two smooth curves.

\begin{nota}
\label{nota:beta}
Choose $\cx_0\in\Sigma$. 
\begin{enumerate}[label=\emph{\roman*)}]
\item $\BB^{\cx_0}$ denotes the second order Taylor expansion of the magnetic field at $\cx_0$. So 
$$\BB^{\cx_0}(\cx) = \tfrac12 (\cx-\cx_0)\Hess\BB(\cx_0)(\cx-\cx_0)^\top.$$ 
\item $\AAA^{\cx_0}$ denotes the Taylor expansion of the magnetic potential $\AAA$ to the third order at the point $\cx_0$. Thus $\BB^{\cx_0}=\nabla\times\AAA^{\cx_0}$.
\item Let $\alpha(\cx_0)$, $\beta(\cx_0)$ be the eigenvalues of $\frac12\Hess\BB(\cx_0)$, agreeing that $|\alpha(\cx_0)|\le |\beta(\cx_0)|$. Set
\[ 
   \varepsilon(\cx_0)=\sqrt{{|\alpha(\cx_0)|}/{|\beta(\cx_0)|}} 
   \quad\mbox{and}\quad
   \Xi(\cx_0)=|\beta(\cx_0)| \,,\]
so that in a suitable local system of orthogonal coordinates $\cy=(s,t)$ centered at $\cx_0$
\[
   \BB^{\cx_0}(\cx) = \alpha(\cx_0)s^2+\beta(\cx_0)t^2 = -\beta(\cx_0) \big(\varepsilon^2(\cx_0)s^2-t^2\big)\,.
\]
\end{enumerate}
\end{nota}
Hence the zero set of $\BB^{\cx_0}$ has the equation $\varepsilon^2(\cx_0)s^2-t^2=0$: It is the union of the two lines $\{t=\pm \varepsilon(\cx_0)s\}$. These lines are the two tangents to the set $\Gamma$ at the point $\cx_0$.

\medskip
The main novelty in this paper is related to the presence of $\Sigma$ and to the role of the following family of model operators indexed by $\varepsilon$ and acting on $L^2(\R^2)$, defined as
\begin{equation}
\label{eq:Xeps}
   \Cr=\Big(D_\sigma-\frac{\tau^3}{3}+\varepsilon^2 \sigma^2 \tau\Big)^2+D_\tau^2,
   \qquad (\sigma,\tau)=:\cY\in\R^2\,,
\end{equation} 
with $D_\sigma = -i\partial_\sigma$ and $D_\tau = -i\partial_\tau$.
Note that the magnetic field associated with the operator $\Cr$ is $\BB(\sigma,\tau)=-\tau^2+\varepsilon^2\sigma^2$. That is why  the operator $\Cr$ for $\varepsilon=\varepsilon(\cx_0)$ will serve as a model magnetic operator at point $\cx_0$.

As a consequence of \cite{HM88,HN}, the operator $\Cr$ has a compact resolvent in $\R^2$. Then it follows that the eigenfunctions of $\Cr$ have an exponential decay (see \cite{Ag82}, \cite{Ag85}, \cite[Theorem B.5.1]{FH} and the example in \cite[p. 100]{theseR}). To sum up:

\begin{prop}\label{2.16aux}
Let $\varepsilon>0$. The spectrum of the operator $\Cr$ is formed by a non-decreasing unbounded sequence of positive eigenvalues denoted by $(\varkappa_{n}(\varepsilon))_{n\geq 1}$.
Moreover, for any eigenfunction $\Psi_{\varepsilon}$ of $\Cr$, there exists $c>0$ such that $\re^{c|Y|}\Psi_{\varepsilon}\in L^2(\R^2)$.
\end{prop}

\subsection{Semiclassical expansions of the magnetic eigenvalues}

The model operators $\Cr$ have the following homogeneity property, due to their magnetic potential of degree $3$. By rescaling, we find immediately:

\begin{lem}
\label{lem:scal}
Set $\AAA_\varepsilon:= (-\frac{\tau^3}{3}+\varepsilon^2 \sigma^2 \tau,0)$ the magnetic potential of $\Cr$. Let $(\varkappa,\Psi)$ be a normalized eigenpair of $\Cr$. Setting for $h>0$ and $\Xi>0$
\[
   \psi_h(\cy) = \Xi^{1/4}h^{-1/4}\Psi(\Xi^{1/4}h^{-1/4}\cY)
\]
we obtain that $(h^{3/2}\,\Xi^{1/2}\varkappa,\,\psi_h)$ is a normalized eigenpair for the semiclassical magnetic operator $(-ih\nabla+\Xi\AAA_\varepsilon)^2$ on $\R^2$.
\end{lem} 
After the scale $h$ for non-vanishing magnetic fields, the scale $h^{4/3}$ for magnetic field vanishing at order 1, we note the apparition of the new scale $h^{3/2}$.

Since for each crossing point $\cx_0\in\Sigma$ the operator $(-ih\nabla+\Xi(\cx_0)\AAA_{\varepsilon(\cx_0)})^2$ is unitarily equivalent to the ``tangent'' magnetic operator $(-ih\nabla+\AAA^{\cx_0})^2$, we can guess  that the behavior of the low lying spectrum of the operator $\Op$ corresponds to the low lying spectrum of the operator
\begin{equation}
\label{eq:OpSigma}
  \mathcal{P}_{h}^{\AAA,\Sigma} := \underset{\cx\in \Sigma}{\bigoplus} 
   \big(-ih\nabla+\Xi(\cx)\AAA_{\varepsilon(\cx)}\big)^2,
\end{equation}
on $L^2(\R^2)^{\sharp \Sigma}$, where $\sharp \Sigma$ is the cardinal of the finite set $\Sigma$. Lemma \ref{lem:scal} then gives that
\[
   \Spec(\mathcal{P}_{h}^{\AAA,\Sigma}) = 
   \underset{\cx\in \Sigma}{\coprod} \:h^{3/2}\,\Xi(\cx)^{1/2}  \Spec(\mathcal{X}_{\varepsilon(\cx)}).
\]
This leads to introduce, for all $\cx\in \Sigma$ and all $n\in\mathbb{N}^*$, the enumeration of the eigenvalues
\begin{subequations}
\begin{equation}
\label{eq:Lam1}
   \Lambda_n^{\cx}=\Xi(\cx)^{1/2}\varkappa_n(\varepsilon(\cx))\,,
\end{equation} 
and the ordered set 
\begin{equation}
\label{eq:Lam2}
   \left\{\Lambda_\mathfrak{n}^{\BB }, \mathfrak{n} \in\mathbb{N}^*\right\}=
   \underset{\cx\in \Sigma}{\coprod}\left\{\Lambda_n^{\cx}, n\in\mathbb{N}^*\right\},
\end{equation}
\end{subequations}
for which the same value can possibly appear several times (for instance, if we have $\Lambda_n^{\cx}=\Lambda^{\cx'}_{n'}$ for $(n,\cx)\neq (n',\cx')$). Note in particular that the smallest element of this set is given by
\begin{equation}
\label{eq:CBSigma}
   \CBSigma=\min_{\cx\in\Sigma} \ \Xi(\cx)^{1/2}\varkappa_1(\varepsilon(\cx))\,,
\end{equation}
where we recall that $\varkappa_1(\varepsilon)$ is the first eigenvalue of $\Cr$ and $\varepsilon(\cx)$ is defined in Notation \ref{nota:beta}.
 
We are ready to state the main two results relating to the low lying spectrum of the operator $\Op$. The first result provides a localized estimate from below of the energy functional $\Fq$ (defined in \eqref{QF}) and exponential decay estimates for the eigenvectors of $\Op$. The second result states an asymptotic expansion for the eigenvalues of $\Op$.

\subsubsection{Estimate from below of the energy and Agmon estimates}
Let the function $\cx\mapsto \dist(\cx,\Sigma)$ be the Euclidean distance between the point $\cx\in \overline{\Omega}$ and the set $\Sigma$.

\begin{thm}\label{thmequcroixINTRO}
Under Assumption \textup{\ref{hypcadre2}}, for any exponent $d\in(\frac{3}{16},\frac{1}{4})$, there exist $\mathrm{c}_d,\:\mathrm{C}_d>0$, $h_{d}>0$, such that for all $u\in \Dom(\Fq)$ and all $h\in(0,h_{d})$,
\begin{subequations}
\begin{equation}
\label{eq:bfb}
   \Fq(u)\ge \int_{\Omega} \mathcal{I}_{h,d}^{\Sigma}(\cx) \, |u(\cx)|^2\,\rd\cx,
\end{equation}
with
\begin{equation}
\label{eq:Ibfb}
   \mathcal{I}_{h,d}^{\Sigma}(\cx)
   =\left\{\begin{array}{ll}
   \CBSigma h^{3/2} - \mathrm{C}_d\, h^{\min\{2-2d,\,3/4+4d\}} &\mbox{if}\quad \dist(\cx,\Sigma)\le  h^{d} 
   \\[0.5ex]
   \mathrm{c}_d \,h^{4/3+2d/3} &\mbox{if}\quad \dist(\cx,\Sigma)> h^{d}\,,
   \end{array}\right.
\end{equation}
\end{subequations}
where $\CBSigma$ is defined in \eqref{eq:CBSigma}. Note that the exponents present in the definition of $\mathcal{I}_{h,d}^{\Sigma}$ satisfy
\begin{equation}
\label{eq:Iexp}
   \min\{2-2d,\,\tfrac34+4d\}>\tfrac32 \quad\mbox{and}\quad \tfrac43+\tfrac{2d}{3}<\tfrac32\,.
\end{equation}
\end{thm}

Agmon estimates are an \textit{a priori} result of exponential decay 
of the eigenfunctions. The following theorem states that the eigenfunctions associated with eigenvalues of order $h^{4/3+2d/3}$ are localized near $\Sigma$.

\begin{thm}\label{agmon_gammad}
Let $L>0$ and $d\in(\frac{3}{16},\frac{1}{4})$. There exist $C$, $h_{d}>0$ such that for all $h\in(0,h_{d})$, and all eigenpair $(\lambda(h),\psi_{h})$ of $\Op$ with $\lambda(h)\le L h^{3/2}$, we have
\begin{equation}
\label{eq:Ag}
   \int_{\Omega} \re^{2\dist(\cx,\Sigma)/h^d}|\psi_{h}(\cx)|^2\,\rd\cx +
   h^{-3/2} \Fq\big(\re^{\dist(\cdot,\Sigma)/h^d}\psi_{h}\big)
   \le C\| \psi_{h} \|^2\,.  
\end{equation}
\end{thm}
\noindent
In fact estimate \eqref{eq:Ag} would also hold for the relaxed condition $\lambda(h)\le L h^{4/3+2d/3}$ if $L<\mathrm{c}_d$.

\subsubsection{Expansion of lowest eigenvalues}
Let us recall that
\[
   \lambda_1(h)\le\lambda_2(h)\le\ldots\le\lambda_n(h)\ldots
\]
denote the increasing sequence of the eigenvalues of the operator $\Op$ while
\[
   \Lambda_1^{\BB}\le\Lambda_2^{\BB}\le\ldots\le\Lambda_n^{\BB}\ldots
\]
denote the elements defined in \eqref{eq:Lam1}-\eqref{eq:Lam2} from the eigenvalues of the model operators $\Cr$.

\begin{thm}\label{thm:asy1}
Under Assumption \textup{\ref{hypcadre2}}, for all $N\in \mathbb{N}^*$, there exist $C_N>0$ and $h_0>0$ such that, for all $h\in (0,h_0)$ and all $n=1,\ldots,N$
\[
   |\lambda_n(h)-h^{3/2}\Lambda_n^{\BB}| \le C_N h^{7/4}\,.
\]
\end{thm}

In Section \ref{sec.3}, an extended version of this theorem is proved, providing full expansions of all lowest eigenvalues $\lambda_n(h)$, see Theorem \ref{thm:asy}.

\subsection{Low lying eigenvalues of the magnetic cross in the small angle limit}
When $\varepsilon$ tends to $0$, the angle between the two lines $\{\tau=\pm \varepsilon\sigma\}$ tends to $0$. It is interesting to understand the behavior of $\varkappa_{n}(\varepsilon)$ in such a limit. One could naively expect that, when $\varepsilon$ goes to $0$, $\inf\Spec(\Cr)$ goes to $\inf \Spec(\mathcal{M}^{[2]})$ where 
\[\mathcal{M}^{[2]}=D^2_{\tau}+\left(D_{\sigma}-\frac{\tau^3}{3}\right)^2\quad\text{acting on} \ L^2(\R^2).\]
The operator $\mathcal{M}^{[2]}$ is sometimes called Montgomery operator of order two. By Fourier transform, we have
\[
   \inf \Spec(\mathcal{M}^{[2]})=\inf_{\xi\in\R}\left(\inf\Spec(\mathcal{M}^{[2]}_{\xi})\right)\,,\quad 
   \text{where} \  \mathcal{M}^{[2]}_{\xi}=D^2_{\tau}+\left(\xi-\frac{\tau^3}{3}\right)^2\ \text{acting on} \ L^2(\R).\]
The quantity $\inf \Spec(\mathcal{M}^{[2]})$ has been numerically estimated, see \cite[Table 1]{HP10}.

Actually, the limit $\varepsilon\to 0$ is singular: The operator $\Cr$ is partially semiclassical.   
Indeed, after the scaling $\sigma=\varepsilon^{-1}s, \ \tau=t$, the operator $\Cr$ becomes 
\begin{equation}\label{eq:opsemicl}
   D_t^2+\Big(\varepsilon D_s-\frac{t^3}{3}+ s^2 t\Big)^2
   \quad\mbox{with}\quad D_t = -i\partial_t,\quad D_s = -i\partial_s.
\end{equation}
The operator \eqref{eq:opsemicl} is the Weyl quantization $\mathsf{Op}^{\sf w}_{\varepsilon}(\sX_{\alpha, \xi})$ of the symbol $\sX_{\alpha, \xi}$ where
\begin{equation}
\label{eq:Xalxi}
   \sX_{\alpha, \xi}=D_t^2+\Big(\xi-\frac{t^3}{3}+ \alpha^2 t\Big)^2
\end{equation}
is a self-adjoint operator acting on $L^2(\R)$ depending on the two real parameters $\alpha$ and $\xi$. In the small angle regime, the spectral analysis of the operator $\Cr$ is related to the spectral analysis of the family $\left(\sX_{\alpha, \xi}\right)_{(\alpha, \xi)\in\R^2}$. The asymptotic expansions of the first eigenvalues of $\Cr$ is related to the``band function'' $\varrho_1$ defined by the ground state energy of $\sX_{\alpha, \xi}$
\[
   \varrho_1(\alpha,\xi)=\min \Spec (\sX_{\alpha,\xi})\,,\quad (\alpha,\xi)\in\R^2\,,
\] 
see for instance \cite{R13} and \cite{RBH}, where such operators and reductions are considered. One of the requirements to apply the theory developed in \cite{RBH} (and that relates $\varrho_1$ to the eigenvalue asymptotic expansions when $\varepsilon\to0$ for the operator $\Cr$) is that $\varrho_1$ has a minimum. 

\begin{thm}\label{thmprincipalsymboleIntro}
The function $\varrho_1$ reaches its infimum $\rS_0$ in $\R^2$, so
\begin{equation*}
   \rS_0=\min_{(\alpha,\xi)\in\R^2}\varrho_1(\alpha,\xi),
\end{equation*}
and this minimum is reached on the set
\begin{equation*}
\Big\{(\alpha,\xi)\in\R^2, \quad |\xi|\le\frac{2}{3}|\alpha|^3 \Big\}\,.
\end{equation*}
\end{thm}
By construction, $\rS_0\le\inf \Spec(\mathcal{M}^{[2]})$. The numerical simulations that we have performed provide the upper bound 
\begin{equation}
\label{eq:ubS_0}
   \rS_0\le \varrho_1(\alpha_0,0) \simeq 0.4941\quad\mbox{for}\quad\alpha_0=0.786.
\end{equation}
We note that the value given in \cite[Table 1]{HP10} for $\inf \Spec(\mathcal{M}^{[2]})$ is $\simeq 0.66$, corresponding to the value $\xi=0$ of the Fourier parameter. Hence 
we have obtained the strict inequality $$\rS_0<\inf \Spec(\mathcal{M}^{[2]}).$$
Our numerical results for the band function $\varrho_1$, see Figures \ref{F1} and \ref{F2}, suggest that the minimum $\rS_0$ is attained at the sole two values $\pm(\alpha_0,0)$ of $(\alpha,\xi)$.

The following result gives the convergence of the eigenvalues $\varkappa_{n}(\varepsilon)$ of the model operator $\Cr$ as $\varepsilon\to0$.
\begin{thm}\label{1.8Intro}
For all $n\ge1$, there exist $C>0$ and $\varepsilon_0>0$ such that for all $\varepsilon\in (0,\varepsilon_0)$
\begin{equation*}|\varkappa_{n}(\varepsilon)-\rS_0|\le C\varepsilon.\end{equation*}
\end{thm}

\subsection{Organization of the paper}
In Section \ref{SA}, we prove the preliminary Theorems \ref{thmequcroixINTRO} and \ref{agmon_gammad}. In Section \ref{sec.3}, we establish the full asymptotic expansions of the eigenvalues and eigenfunctions. Section \ref{PSA} is devoted to the proofs of Theorems \ref{thmprincipalsymboleIntro} and \ref{1.8Intro} and to the presentation of numerical simulations concerning the band function $\varrho_1$ and the eigenpairs of $\Cr$ as $\varepsilon\to0$.

\section{Bounds from below and exponential decay}\label{SA}
In this section we prove the preliminary bounds from below \eqref{eq:bfb}-\eqref{eq:Ibfb} and the Agmon estimates \eqref{eq:Ag}. We will use several times a perturbation formula for the magnetic Laplacian which we first state.

\subsection{Perturbation of the magnetic potential}
Let $\Astar$ be another smooth magnetic potential defined on $\overline\Omega$. In practice $\Astar$ will be the Taylor expansion of $\AAA$ at some point $\cx_0$ and to various orders (2, 3 or 4). By expanding the square, we get
\begin{equation}
\label{eq:diff}
  \Fq(u) =   
  \Fqstar(u)
  + 2 \Re \big\langle (-ih\nabla+\Astar)u,(\AAA-\Astar)u\big\rangle 
  + \|(\AAA-\Astar)u\|^2.
\end{equation}
This yields
$
  \Fq(u) \geq \Fqstar(u) - 2  \Fqstar(u)^{1/2} \|(\AAA-\Astar)u\|
$
by Cauchy-Schwarz inequality, leading to the parametric estimate
(based on the inequality $2ab\leq \eta a^2+\eta^{-1}b^2$, for all $\eta>0$)
\begin{equation}\label{eq:minoeta}
   \forall \eta>0, \quad 
   \Fq(u) \geq   (1-\eta) \Fqstar(u) - \eta^{-1}\|(\AAA-\Astar)u\|^2\,.
\end{equation}
Such a lower bound is used, for instance, in the seminal paper \cite[p. 51]{HM96}, and in the book \cite[Chap. 8]{FH}.

\subsection{Lower bound}
The proof of \eqref{eq:bfb} is based on a quadratic partition of unity. Let us recall that such a partition is given for each relevant $h>0$ by a finite collection of smooth cutoff functions $(\chi_{j}^{h})$ satisfying on $\Omega$
\begin{equation}\label{partition}
   \sum_j |\chi_{j}^{h}|^2=1\,. 
\end{equation}
Then we have the following well-known localization formula (see \cite{CFKS86})
\begin{equation}
\label{formule de localisation}
   \Fq(u)=\sum_j\Fq(\chi_{j}^{h}u)-h^2\sum_j\| u\,|\nabla(\chi_{j}^{h})|\|^2_{L^2(\Omega)}\,.
\end{equation}

We introduce three sets $\Sigma^{[1]}$, $\Sigma^{[2]}(h)$ and $\Sigma^{[3]}(h)$ covering $\Omega$. 
\begin{nota}\label{notadecoupagebis}
Let $\cB(\cx_0,r)$ denote the open ball of center $\cx_0$ and radius $r$ and
\[
   \cB(\Sigma,r) = \bigcup_{\cx\in\Sigma} \cB(\cx,r).
\]
Let $R_1>0$ such that $\cB(\Sigma,3R_1)\subset\Omega$. Choose $d\in(\frac{3}{16},\frac{1}{4})$ and introduce
\begin{enumerate}[label=(\roman*)]
\item $\Sigma^{[1]} = \Omega\setminus\overline\cB(\Sigma,R_1)$.
\item $\Sigma^{[2]}(h) = \cB(\Sigma,2R_1)\setminus\overline\cB(\Sigma,\frac12 h^{d})$.
\item $\Sigma^{[3]}(h) = \cB(\Sigma,h^{d})$.
\end{enumerate}
\end{nota}

Then we consider a partition of unity composed of three cutoff functions 
$(\chifrakUU,\mathfrak{X}^{h,[2]},\mathfrak{X}^{h,[3]})$ associated with this covering of $\Omega$ in the sense that
\begin{equation*}
 |\chifrakUU|^2+|\mathfrak{X}^{h,[2]}|^2+|\mathfrak{X}^{h,[3]}|^2=1,
\end{equation*}
and
\begin{equation*}
\supp \chifrakUU \subset \Sigma^{[1]},\quad \supp \mathfrak{X}^{h,[2]} \subset \Sigma^{[2]}(h),\quad \supp \mathfrak{X}^{h,[3]} \subset \Sigma^{[3]}(h).
\end{equation*}
The distance between $\partial\Sigma^{[2]}(h)\cap\Sigma^{[1]}(h)$ and $\partial\Sigma^{[1]}(h)\cap\Sigma^{[2]}(h)$ is $R_1$. The distance between $\partial\Sigma^{[3]}(h)$ and $\partial\Sigma^{[2]}(h)\cap\Sigma^{[3]}(h)$ is $\frac12 h^d$. Hence we can choose the cutoff functions so that
\begin{equation}
\label{eq:nabla}
|\nabla\chifrakUU|^2\le \Cims, \quad 
|\nabla\mathfrak{X}^{h,[2]}|^2\le \Cims h^{-2d}, \quad 
|\nabla\mathfrak{X}^{h,[3]}|^2\le \Cims h^{-2d}.
\end{equation}
with a constant $\Cims$ independent of $h$ and $d$.
Combined with the localization formula \eqref{formule de localisation} associated with the partition $(\chifrakUU,\mathfrak{X}^{h,[2]},\mathfrak{X}^{h,[3]})$, the estimates \eqref{eq:nabla} yield  for all $u \in H^1(\Omega)$, 
\begin{equation}\label{eq:1b}
   \Fq(u)\ge   \Fq(\chifrakUU u)+\Fq(\mathfrak{X}^{h,[2]} u)+\Fq(\mathfrak{X}^{h,[3]} u)
   -\Cims h^{2-2d}\| u \|^2_{L^2(\Omega)},
\end{equation}
We set  
$u_1=\chifrakUU u$, $u_2=\mathfrak{X}^{h,[2]}u$ and $u_3=\mathfrak{X}^{h,[3]}u$, and are going to bound each $\Fq(u_k)$ from below (for $k=1,2,3$).

\begin{lem}[Lower bound on $\Sigma^{[1]}$]\label{lem:eq:2}
There exist $c_1$ and $h_0>0$ such that for all $0<h<h_0$
\begin{equation*}
   \Fq(u_1) \ge c_1 h^{4/3} \|u_1\|^2\,.
\end{equation*}
\end{lem}
\begin{proof} 
This result is a consequence of \cite[Theorem 1.9]{artJP}. 
\end{proof}

\begin{lem}[Lower bound on $\Sigma^{[2]}(h)$]\label{lem:eq:8b}
There exist $c_2$ and $h_0>0$ such that for all $0<h<h_0$
\begin{equation*}
   \Fq(u_2)  \ge c_2 h^{4/3+2d/3} \|u_2\|^2\,.
\end{equation*}
\end{lem}

\begin{proof}
Set $\rho=\frac14$.
We introduce a second partition of unity $(\chi_j^h)_{j\in\mathcal{J}}$ on $\Sigma^{[2]}(h)$ associated with a family of balls $\cB(\cx_j, d_jh^\rho)$ for some constants $d_j$. We first cover the zero set $\Gamma\cap\Sigma^{[2]}(h)$ and choose the centers $\cx_j\in\Gamma$ with $d_j=1$ so that the distances between consecutive points along $\Gamma$ is $\simeq\frac12 h^\rho$. Hence, setting
\[
   \widetilde{\Sigma^{[2]}}(h) = \Sigma^{[2]}(h) \setminus \bigcup_{\cx_j\in\Gamma} \cB(\cx_j,h^\rho)
\]
we obtain that the distance between $\Gamma$ and $\widetilde{\Sigma^{[2]}}(h)$ is larger than $\frac12 h^\rho$. We then cover $\widetilde{\Sigma^{[2]}}(h)$ with balls $\cB(\cx_j, d_jh^\rho)$ choosing $d_j=\frac14$ and the mutual distances between the centers bounded from below by $\frac18h^\rho$. Finally we can choose the functions $\chi_j^h$ so that $|\nabla\chi_j^h|^2\le C_{2,\IMS} h^{2\rho}$ and the localization formula \eqref{formule de localisation} yields
\begin{equation}
\label{eq:3}
   \Fq(u_2) \ge \sum_j \Fq(\chi_j^h u_2) - C_{2,\IMS} h^{2-2\rho}\|u_2\|^2\,.
\end{equation}
We have, by \cite[Lemma 2.3]{artJP} (for the case $(\ell)=(3)$ of Table 2.1),
\begin{equation}
\label{eq:4b}
   \Fq(\chi_j^h u_2) \ge \Big(\frac12 \mathrm{M}_{0} \, |\nabla\BB(\cx_j)|^{2/3} h^{4/3} - 2C h^{6\rho}\Big) 
   \|\chi_j^h u_2\|^2 \quad\quad (\cx_j\in\Gamma).
\end{equation}
As a consequence of the non degeneracy of $\nabla\BB$ on $\Gamma$ outside $\Sigma$ (namely $\nabla \BB|\neq 0$ on $\Gamma\backslash \Sigma$) and the non degeneracy of $\Hess\BB$ on $\Sigma$ (meaning that the eigenvalues of the Hessian matrix are not equal to $0$, according to Assumption \ref{hypcadre2})
there exists a positive constant $C(\BB)$ such that 
\[|\nabla\BB(\cx)|\ge C(\BB)\, \dist(\cx,\Sigma),\quad \forall\cx\in\Omega\,.\] 
Since $\dist(\cx_j,\Sigma)$ is larger than $\frac12 h^d$ by construction, from \eqref{eq:4b} we get
\begin{equation*}
   \Fq(\chi_j^h u_2) \ge \Big(\frac14 \mathrm{M}_{0}  C(\BB)^{2/3} h^{4/3+2d/3} - 2C h^{6\rho}\Big) 
   \|\chi_j^h u_2\|^2.
\end{equation*}
We note that, with $\rho=\frac14$ and $d<\frac14$, the exponent $6\rho=\frac32$ is (strictly) larger than $\frac43+\frac23d$. Therefore, for $h$ small enough
\begin{equation}
\label{eq:5b}
   \Fq(\chi_j^h u_2) \ge \frac18 \mathrm{M}_{0}  C(\BB)^{2/3} h^{4/3+2d/3}  
   \|\chi_j^h u_2\|^2 \quad\quad (\cx_j\in\Gamma).
\end{equation}
When $\cx_j$ does not belong to $\Gamma$, in the ball $\cB(\cx_j,d_jh^\rho)$ we use the lower bound (see \cite[Lemma 1.4.1]{FH})
\begin{equation}
\label{eq:4bb}
   \Fq(\chi_j^h u_2) \ge \inf_{\cx\in\cB(\cx_j,d_jh^\rho)} |\BB(\cx)| \: h \:
   \|\chi_j^h u_2\|^2 \quad\quad (\cx_j\not\in\Gamma).
\end{equation}
For any $\cx\in\cB(\cx_j,d_jh^\rho)$, its distance to $\Gamma$ is larger than $\frac12h^\rho$ by construction. Let $\cg\in\Gamma$ be such that $\dist(\cx,\Gamma)=\dist(\cx,\cg)$. Then $\dist(\cg,\Sigma)$ is larger than $\frac12h^d$. As a consequence of the Morse lemma
\[
   |\BB(\cx)| \ge C'(\BB)\,\dist(\cx,\cg)\,\dist(\cg,\Sigma) \ge C h^{\rho+d}.
\]
Hence, with \eqref{eq:4bb} we obtain
\begin{equation}
\label{eq:5bb}
   \Fq(\chi_j^h u_2) \ge C h^{1+\rho+d} \:
   \|\chi_j^h u_2\|^2 \quad\quad (\cx_j\not\in\Gamma).
\end{equation}
With $\rho=\frac14$ and $d<\frac14$, the exponent $1+\rho+d$ is $<\frac43+\frac23d$. Therefore, for $h$ small enough, as a result of \eqref{eq:5b} and \eqref{eq:5bb} we have, for a positive constant $c$ independent of $h$
\begin{equation*}
   \Fq(\chi_j^h u_2) \ge c h^{4/3+2d/3}  
   \|\chi_j^h u_2\|^2 \quad\quad (\forall j\in\mathcal{J}).
\end{equation*}
With \eqref{eq:3} this yields
\begin{align*}
   \Fq(u_2) &\ge c h^{4/3+2d/3} \sum_j \|\chi_j^h u_2\|^2 - C_{2,\IMS} h^{2-2\rho}\|u_2\|^2\\
   &\ge c_2 h^{4/3+2d/3} \|u_2\|^2 \,,
\end{align*}
since $2-2\rho=\frac32>4/3+2d/3$ for any $d<\frac14$. The lemma is proved.
\end{proof}

\begin{lem}[Lower bound on $\Sigma^{[3]}$]\label{lem:eq:12}
There exist $C_3$ and $h_0>0$ such that for all $0<h<h_0$
\begin{equation*}
   \Fq(u_3) \ge \big(\CBSigma  h^{3/2} - C_3h^{3/4+4d}\big) \|u_3\|^2.
\end{equation*}
\end{lem}

\begin{proof}
For $h$ small enough the set $\Sigma^{[3]}$ is the union of the balls $\cB(\cx_0,h^d)$, with $\cx_0$ spanning $\Sigma$. It suffices to consider each point $\cx_0$ separately. We denote by $u^{\cx_0}_3$ the restriction of $u_3$ to $\cB(\cx_0,h^d)$ and use the perturbative lower bound \eqref{eq:minoeta} with $\Astar=\AAA^{\cx_0}$ the third order Taylor expansion of $\AAA$ at $\cx_0$:
\[
   \forall \eta>0, \quad 
   \Fq(u^{\cx_0}_3) \geq   (1-\eta) \mathcal{Q}_{h,\Omega}^{\AAA^{\cx_0}}(u^{\cx_0}_3)
    - \eta^{-1}\|(\AAA-\AAA^{\cx_0})u^{\cx_0}_3\|^2\,.
\]
Recall from the introduction that the magnetic field associated with $\AAA^{\cx_0}$ is $\BB^{\cx_0}$, and that the  magnetic operator $(-ih\nabla+\AAA^{\cx_0})^2$ is unitarily equivalent to $(-ih\nabla+\Xi(\cx_0)\AAA_{\varepsilon(\cx_0)})^2$  whose first eigenvalue is $h^{3/2}\Lambda^{\cx_0}_1$. Besides, 
\[
   |\AAA(\cx)-\AAA^{\cx_0}(\cx)| \le C h^{4d},\quad\quad \forall\cx\in\cB(\cx_0,h^d).
\]
So we have
\[
   \forall \eta>0, \quad 
   \Fq(u^{\cx_0}_3) \geq   \big((1-\eta) h^{3/2}\Lambda^{\cx_0}_1
    - \eta^{-1}Ch^{8d}\big)\|u^{\cx_0}_3\|^2\,.
\]
Choosing $\eta=h^{4d-3/4}$ to equalize the remainders, we deduce the inequality $\Fq(u^{\cx_0}_3) \geq \big( h^{3/2}\Lambda^{\cx_0}_1
    - Ch^{3/4+4d}\big)\|u^{\cx_0}_3\|^2$. Taking the infimum over all the points of the finite set $\Sigma$ gives the result.
\end{proof}

We can now conclude with the proof of Theorem \ref{thmequcroixINTRO}.

\begin{proof}[Proof of Theorem \ref{thmequcroixINTRO}]
We gather the estimates provided by Lemmas \ref{lem:eq:2}--\ref{lem:eq:12} and combine them with the localization estimate \eqref{eq:1b} and obtain that $\Fq(u)$ is bounded from below by
\begin{equation*}
   c_1 h^{4/3} \|u_1\|^2 + c_2 h^{4/3+2d/3} \|u_2\|^2 +
   \big(\CBSigma  h^{3/2} - C_3h^{3/4+4d}\big) \|u_3\|^2
   -\Cims h^{2-2d} 
   \|u\|^2.
\end{equation*}
Since $d<\frac14$, then $h^{2-2d}\ll h^{4/3+2d/3}$ and the lower bound above can be replaced by
\begin{multline*}
   c h^{4/3+2d/3} (\|u_1\|^2+\|u_2\|^2) +
   \big(\CBSigma  h^{3/2} - C_3h^{3/4+4d}-\Cims h^{2-2d}\big) \|u_3\|^2 \\
   \ge c h^{4/3+2d/3} \|(1-\chi)u\|^2 + 
   \big(\CBSigma  h^{3/2} - Ch^{\min\{3/4+4d,2-2d\}}\big) \|\chi u\|^2
\end{multline*}
with $\chi$ the characteristic function of the set $\{\cx\in\Omega,\: \dist(\cx,\Sigma)<h^d\}$. The theorem is proved.
\end{proof}

\subsection{Agmon estimates}
Let $d\in(\frac{3}{16},\frac14)$ and $L>0$. We consider an eigenpair $(\lambda(h),\psi_{h})$ of $\Op$ such that $\lambda(h)\leq L h^{3/2}$. For all function $\Phi\in W^{1,\infty}(\Omega)$, we have $\re^{\Phi}\psi_{h}\in \Dom(\Fq)$ and
\begin{equation}\label{agmonIDd}
   \Fq(\re^{\Phi}\psi_{h}) =
\lambda(h)\| \re^{\Phi}\psi_{h}\|^2+h^2\left\||\nabla\Phi| \,\re^{\Phi}\psi_{h}\right\|^2\,.
\end{equation}
Defining $\Phi(\cx)=\dist(\cx,\Sigma)h^{-d}$, we have
\begin{equation*}
   |\nabla\Phi(\cx)|^{2} = h^{-2d}\,.
\end{equation*}
Hence 
\begin{equation}\label{agmonRested}
   h^2 \| |\nabla\Phi|\,\re^{\Phi}\psi_{h}\|^2\le 
   h^{2-2d}\| \re^{\Phi}\psi_{h}\|^2.
\end{equation}
We denote
\[ 
   Z^{\sf near}_h=\{\cx\in\Omega,\  \dist(\cx,\Sigma)\le h^d \} \quad\mbox{and}\quad 
   Z^{\sf far}_h =\{\cx\in\Omega,\  \dist(\cx,\Sigma) > h^d \}\,.
\]
We introduce $\delta$ such that $\delta = \min\{3/4+4d,2-2d\}-\frac32$. Since $d\in(\frac{3}{16},\frac14)$,   the number $\delta$ is positive. Using \eqref{eq:bfb}-\eqref{eq:Ibfb} we obtain
\begin{align*}
   \Fq(\re^{\Phi}\psi_{h}) &\ge 
   \mathrm{c}_d \,h^{4/3+2d/3}\left\| \re^{\Phi}\psi_{h}\right\|^{2}_{L^2(Z^{\sf far}_h)} +
   \left(\CBSigma h^{3/2}-\mathrm{C}_d\,h^{3/2+\delta})\right) 
   \left\| \re^{\Phi}\psi_{h}\right\|^{2}_{L^2(Z^{\sf near}_h)}\\
   &\ge \mathrm{c}_d \,h^{4/3+2d/3}\left\| \re^{\Phi}\psi_{h}\right\|^{2}_{L^2(Z^{\sf far}_h)} 
   -\mathrm{C}_d\,h^{3/2+\delta}
   \left\| \re^{\Phi}\psi_{h}\right\|^{2}_{L^2(Z^{\sf near}_h)}
   \,. 
\end{align*}
Combining the latter inequality with \eqref{agmonIDd} and \eqref{agmonRested}, we obtain (using the fact that $\lambda(h)\leq L h^{3/2}$)
\[
   \left(\mathrm{c}_d \,h^{4/3+2d/3} - Lh^{3/2} - h^{2-2d}\right)
   \| \re^{\Phi}\psi_{h}\|^2_{L^2(Z^{\sf far}_h)} 
   \le (Lh^{3/2} + h^{2-2d} + \mathrm{C}_d\,h^{3/2+\delta}) 
   \| \re^{\Phi}\psi_{h}\|^{2}_{L^2(Z^{\sf near}_h)}
\]
from which we immediately deduce that for $h$ small enough
\[
   \left\| \re^{\Phi}\psi_{h}\right\|^{2}_{L^2(Z^{\sf far}_h)}\le 
   \left\| \re^{\Phi}\psi_{h}\right\|^{2}_{L^2(Z^{\sf near}_h)}\,.
\]
Since by construction $|\Phi|$ is bounded by $1$ on $Z^{\sf near}_h$, there holds
\[
    \left\| \re^{\Phi}\psi_{h}\right\|^{2}_{L^2(Z^{\sf near}_h)}
    \le \left\| \psi_{h}\right\|^{2}_{L^2(Z^{\sf near}_h)}\,,
\]  
we finally get $\| \re^{\Phi}\psi_{h}\|^{2}_{L^2(\Omega)} \le \| \psi_{h}\|^{2}_{L^2(\Omega)}$, which implies the Agmon estimates \eqref{eq:Ag}.
\medskip

\section{Asymptotic expansions of eigenvalues}\label{sec.3}
In this section, we prove Theorem \ref{thm:asy1} in two steps. First, the localization around each crossing point $\cx_0\in\Sigma$, and, second, a perturbation argument for the magnetic potential around each crossing point. We conclude the section by stating a full asymptotic expansion for eigenvalues and a sketch of the proof.

\subsection{Preliminaries}
We will use several times an argument based on reciprocal quasimodes between two operators. We state this in a general form.

\begin{lem}
\label{lem:PPstar}
Let $\cP$ and $\cP\sta$ two positive operators with discrete spectra associated with sesqui\-linear forms $a$ and $a\sta$, respectively. Let $(\mu_n)$ and $(\mu\sta_n)$ be the increasing sequences of their eigenvalues (counted with multiplicity). Let $(\varphi_n)$ be an associated orthonormal basis of eigenvectors for $\cP$. Let $N$ be a positive integer and assume that for each $n=1,\ldots,N$, there exists $\varphi\sta_n\in\Dom(a\sta)$ such that
\begin{subequations}
\begin{equation}
\label{eq:fista}
   \langle \varphi\sta_n,\varphi\sta_m\rangle =  \langle \varphi_n,\varphi_m\rangle + \nu_{n,m}
   \quad\mbox{and}\quad
   a\sta(\varphi\sta_n,\varphi\sta_m) =   a(\varphi_n,\varphi_m) + \mu_{n,m}
\end{equation}
and set
\begin{equation}
\label{eq:numu}
   \nu = \max_{n,m} |\nu_{n,m}|\quad\mbox{and}\quad \mu = \max_{n,m} |\mu_{n,m}|
\end{equation}
\end{subequations}
Assume that $\nu<\frac{1}{N}$. Then
\begin{equation}
\label{eq:PPstar}
   \mu\sta_n \le \frac{\mu_n + n\mu}{1-n\nu}, \quad n=1,\ldots,N.
\end{equation}
\end{lem}

\begin{proof}
Let $M\le N$.
By the min-max formula, 
\[
   \mu\sta_M \le \max_{\varphi\in\spa\{\varphi\sts_n,\,n=1,\ldots,M\}}
   \frac{a\sta(\varphi,\varphi)}{\langle\varphi,\varphi\rangle}
\]
We write $\varphi$ as a sum $\sum_n\gamma_n\varphi\sta_n$. Then
\[
\begin{aligned}
   a\sta(\varphi,\varphi) &= \sum_{n,m} \gamma_n\bar\gamma_m a\sta(\varphi\sta_n,\varphi\sta_m) \\
   &= \sum_n|\gamma_n|^2 \mu_n + \sum_{n,m}\gamma_n\bar\gamma_m\mu_{n,m}
   \le (\mu_M + M \mu ) \sum_n|\gamma_n|^2.
\end{aligned}
\]
Likewise
\[
\begin{aligned}
   \langle\varphi,\varphi\rangle &= \sum_{n,m} \gamma_n\bar\gamma_m \langle\varphi\sta_n,\varphi\sta_m\rangle \\
   &= \sum_n|\gamma_n|^2 + \sum_{n,m}\gamma_n\bar\gamma_m\nu_{n,m}
   \ge (1 - M \nu ) \sum_n|\gamma_n|^2.
\end{aligned}
\]
Whence formula \eqref{eq:PPstar}
\end{proof}

We will also need a simple but useful consequence of Agmon estimates.

\begin{lem}
\label{lem:Ag+}
Assume that the family of function $(\psi_h)_{h>0}$ satisfies (for some $\gamma$ and $\delta>0$) the estimate
\[
   \int_{\R^2} \re^{\gamma|\cy|/h^\delta}|\psi_{h}(\cy)|^2\,\rd\cy
   \le C\| \psi_{h} \|^2\,,  
\]
for $h$ small enough and $C$ independent of $h$. Then for all $m>0$, there exists $C_m$ such that for $h>0$ small enough
\[
   \int_{\R^2} |\cy|^{m}\,|\psi_{h}(\cy)|^2\,\rd\cy
   \le C_m h^{m\delta}\,\| \psi_{h} \|^2\,.  
\]
\end{lem}

\begin{proof}
It suffices to write
\[
   \int_{\R^2} |\cy|^{m}\,|\psi_{h}(\cy)|^2\,\rd\cy \le
   \max_{\rho>0} \rho^{m} \re^{-\gamma\rho/h^\delta}
   \int_{\R^2} \re^{\gamma|\cy|/h^\delta}\,|\psi_{h}(\cy)|^2\,\rd\cy
\]
and notice that $\max_{\rho>0} \rho^{m} \re^{-\gamma\rho/h^\delta} = \max_{\rho>0} (h^\delta\rho)^{m} \re^{-\gamma\rho} =
h^{m\delta} \max_{\rho>0} \rho^{m} \re^{-\gamma\rho}$.
\end{proof}

\subsection{Localization}
Introduce a positive radius $r_0$ such that the collection of balls $\Omegazz$, $\cx_0\in\Sigma$, are pairwise disjoint. Then consider the collection of operators $\Oploc$ for $\cx_0\in\Sigma$. The main result of this section is:

\begin{lem}
\label{lem:loca}
Recall that $\lambda_n(h)$ is the increasing sequence of eigenvalues of $\Op$ (counted with multiplicity).
Denote by $\lambda\loc_n(h)$ the increasing sequence of eigenvalues of
\[
   \bigoplus_{\cx_0\in\Sigma}\: \Oploc\,.
\]
Let $d\in(\frac{3}{16},\frac14)$ and let $N$ be a positive integer. If $\lambda\loc_N(h)\le Lh^{3/2}$ for $h>0$ small enough, then there exist $C_N$ and $h_N>0$ such that for all $h\in(0,h_N)$
\[
   | \lambda_n(h) - \lambda\loc_n(h)| \le C_N \,\re^{-r_0/h^d},\quad \forall \ n=1,\ldots,N.
\]
\end{lem} 

\begin{proof}
Choose for each $\cx_0\in\Sigma$ a smooth cut-off function $\chi_{\cx_0}$ with support in $\Omegazz$ and equal to $1$ on $\cB(\cx_0,\frac12 r_0)$.

\medskip\noindent
{\em i)} Use Lemma \ref{lem:PPstar} with $\cP\sta=\Op$ and $\cP=\oplus_{\cx_0\in\Sigma}\, \Oploc$ (acting on $\oplus_{\cx_0\in\Sigma}L^2(\cB(\cx_0,r_0))$): For each eigenvector $\varphi_n$ of $\cP$, one crossing point $\cx_0\in\Sigma$ is selected and we set
\[
   \varphi\sta_n = \chi_{\cx_0}\,\varphi_n \quad\mbox{defined on}\quad \Omega.
\]
Relying on the assumption that $\lambda\loc_n(h)\le Lh^{3/2}$ for $n=1,\ldots,N$, we may apply the Agmon estimates \eqref{eq:Ag} to the operator $\Oploc$. This yields conditions \eqref{eq:fista}-\eqref{eq:numu} considering $\nu=\mu=C\re^{-r_0/h^d}$. Hence $\lambda_n(h)\le \lambda\loc_n(h)+C_N \,\re^{-r_0/h^d}$. 

\medskip\noindent
{\em ii)} Conversely, we swap the roles of the two operators: $\cP=\Op$ and $\cP\sta=\oplus_{\cx_0\in\Sigma}\, \Oploc$. For each eigenvector $\varphi_n$ of $\cP$, we consider
\[
   \varphi\sta_n := \big( \chi_{\cx_0}\,\varphi_n\big)_{\cx_0\in\Sigma}\quad\mbox{defined on}\quad
   \coprod_{\cx_0\in\Sigma} \Omegazz\,.
\]
Note that, as a consequence of the previous step of the proof, we have that $\lambda_n(h)\le 2Lh^{3/2}$ for all $n=1,\ldots,N$.
As above, we conclude with the help of Agmon estimates for the operator $\Op$ that $\lambda\loc_n(h)\le \lambda_n(h)+C_N \,\re^{-r_0/h^d}$. The lemma is proved.
\end{proof}

\subsection{Taylor approximation of a localized operator}
With Lemma \ref{lem:loca} at hand, we can assume that $\Omega=\Omegazz$. Thus $\Sigma$ is reduced to one element, $\cx_0$. Recall that $\AAA^{\cx_0}$ denotes the third order Taylor expansion of $\AAA$ at the point $\cx_0$. We are going to consider the operator $\Opzz:=(-ih\nabla+\AAA^{\cx_0})^2$ posed on $\R^2$. We have seen in the introduction that the eigenvalues of $\Opzz$ are the $h^{3/2}\Lambda^{\cx_0}_n$, see \eqref{eq:Lam1}, and that its eigenvectors $\psi^{\cx_0}_{h,n}$ are scaled from the eigenvectors of $\Cr$ with $\varepsilon=\varepsilon(\cx_0)$ and $\Xi=\Xi(\cx_0)$. As a consequence of the exponential decay of the eigenvectors of $\Cr$ (Proposition \ref{2.16aux}) and the scaling provided by Lemma \ref{lem:scal}, we find that, for some positive constant $\gamma$ 
\begin{equation}
\label{eq:Ag14}
   \int_{\R^2} \re^{2\gamma|\cy|/h^{1/4}}|\psi^{\cx_0}_{h,n}(\cy)|^2\,\rd\cy +
   h^{-3/2} \Fqzz\big(\re^{\gamma|\cy|/h^{1/4}}\psi^{\cx_0}_{h,n}\big)
   \le C\| \psi^{\cx_0}_{h,n} \|^2\,.  
\end{equation}

\begin{lem}
\label{lem:tayapp}
With $\Omega=\Omegazz$, we denote by $\lambda^{\cx_0}_n(h)$ the eigenvalues of $\Op$. The eigenvalues of $\Opzz$ are given by $h^{3/2}\Lambda^{\cx_0}_n$. For any $d_0<\frac{1}{4}$ and any positive integer $N$, there exist $C_N$ and $h_N>0$ such that for all $h\in(0,h_N)$
\begin{equation}
\label{eq:tayapp}
    h^{3/2}\Lambda^{\cx_0}_n - C_N \,h^{3/2+d_0} \:\le\:
   \lambda^{\cx_0}_n(h) \:\le\:  h^{3/2}\Lambda^{\cx_0}_n + C_N \,h^{7/4},
   \quad n=1,\ldots,N.
\end{equation}
\end{lem}

\begin{proof}
The proof combines Lemma \ref{lem:PPstar} with the perturbation identity \eqref{eq:diff}. We still use the cut-off function $\chi_{\cx_0}$ as in the proof of Lemma \ref{lem:loca}.

\medskip\noindent
{\em i)} Use Lemma \ref{lem:PPstar} with $\cP=\Opzz$ and $\cP\sta=\Op$. For each eigenvector $\varphi_n$ of $\cP$, we consider the quasimode $\varphi\sta_n = \chi_{\cx_0}\varphi_n$ for $\Opzz$. The localization error is exponentially decaying thanks to \eqref{eq:Ag14} but the principal part of the discrepancy for the quasimodes arises from the difference between $\AAA$ and $\AAA^{\cx_0}$. Thus the bound $\nu$ in \eqref{eq:numu} satisfies (for some $r_0>0$)
\[
   \nu \le C \re^{-r_0/h^{1/4}}.
\]
For estimate the diagonal terms $\mu_{n,n}$ in \eqref{eq:fista}, we use the identity \eqref{eq:diff} for $\Astar=\AAA^{\cx_0}$. Then
\[
\begin{aligned}
   \Fq(\chi_{\cx_0}\psi^{\cx_0}_{h,n}) 
   &= \Fqzz(\chi_{\cx_0}\psi^{\cx_0}_{h,n}) 
   + 2 \Re \big\langle (-ih\nabla+\AAA^{\cx_0})\chi_{\cx_0}\psi^{\cx_0}_{h,n},
   (\AAA-\AAA^{\cx_0})\chi_{\cx_0}\psi^{\cx_0}_{h,n}\big\rangle \\
  &\qquad + \|(\AAA-\AAA^{\cx_0})\chi_{\cx_0}\psi^{\cx_0}_{h,n}\|^2.
\end{aligned}
\]
Hence the difference $\mu_{n,n}:=\Fq(\chi_{\cx_0}\psi^{\cx_0}_{h,n})- \Fqzz(\chi_{\cx_0}\psi^{\cx_0}_{h,n})$ is estimated by
\[
   |\mu_{n,n}| \le 2\big(\Fqzz(\chi_{\cx_0}\psi^{\cx_0}_{h,n})\big)^{1/2} 
   \|(\AAA-\AAA^{\cx_0})\chi_{\cx_0}\psi^{\cx_0}_{h,n}\| +
   \|(\AAA-\AAA^{\cx_0})\chi_{\cx_0}\psi^{\cx_0}_{h,n}\|^2
\]
Using the Agmon estimates \eqref{eq:Ag14} and Lemma \ref{lem:Ag+} with $\delta=\frac14$ and $m=8$, we find 
\[
   |\mu_{n,n}| \le C(h^{3/4} h + h^2) .
\]
The reasonning is similar for $\mu_{n,m}$, $n\neq m$.
Hence the right part of inequalities \eqref{eq:tayapp}.

\medskip\noindent
{\em ii)} For the left part of \eqref{eq:tayapp}, we swap the roles of $\Opzz$ and $\Op$. The sole difference consists in Agmon estimates for the truncated eigenvectors of $\Op$. Thanks to the right part of \eqref{eq:tayapp}, Agmon estimates \eqref{eq:Ag} hold and then Lemma \ref{lem:Ag+} with $\delta=d\in (\frac{3}{16},\frac14)$ and $m=8$, from which we deduce
\[
   |\mu_{n,n}| \le C(h^{3/4} h^{4d} + h^{8d}) .
\]
Choosing $d=\frac{3}{16}+\frac{d_0}{4}$, we obtain the left part of \eqref{eq:tayapp}.
\end{proof}

\subsection{Expansion of eigenvalues}
Putting together Lemmas \ref{lem:loca} and \ref{lem:tayapp}, we can see that Lemma \ref{lem:tayapp} provides the bound $\lambda\loc_n(h)\le Lh^{3/2}$ which validates the application of Lemma \ref{lem:loca}. Therefore, we have now proved:

\begin{lem}
Under Assumption \textup{\ref{hypcadre2}}, for any $d<\frac14$ and any positive integer $N$, there exist $C_N>0$ and $h_0>0$ such that, for all $h\in (0,h_0)$ and all $n=1,\ldots,N$
\begin{equation}
\label{eq:asy1}
    h^{3/2}\Lambda^{\BB}_n - C_N \,h^{3/2+d} \:\le\:
   \lambda_n(h) \:\le\:  h^{3/2}\Lambda^{\BB}_n + C_N \,h^{7/4},
   \quad n=1,\ldots,N.
\end{equation}
\end{lem}

To go further, we are going to exhibit, for each $n$, series expansions of eigenpairs. Owing to the exponential localization given by Lemma \ref{lem:loca}, it suffices to restrict the construction to any chosen localized operator $\Oploc$. In order to alleviate notations, we will remove the mention of $\cx_0$ in general, and work in the Cartesian coordinates $\cy$ for which the crossing point is at the origin. Then the domain $\Omega$ is the ball $\cB(0,r_0)$ and the magnetic field cancels to the order $2$ at $0$. After a possible change of gauge, we can assume that the magnetic potential $\AAA$ cancels to the order $3$ at $0$. We write its Taylor formal series as
\[
   \AAA \sim \sum_{j\ge0} \AAA_j\quad\mbox{where}\quad  
   \mbox{$\AAA_j$ is polynomial and homogeneous of degree $3+j$.}
\]
The first nonzero term is $\AAA_0$ formerly denoted by $\AAA^{\cx_0}$.
We retrieve the principal part of $\Op$ at $0$, and its natural expansion in powers of $h^{1/4}$ by considering, via the change of variables $\cY = \cy/h^{1/4}$:
\[
   h^{-3/2}\Op[\cy,\nabla_\cy] =  
   h^{-3/2} \mathcal{P}_{h,\Omega/h^{1/4}}^{\AAA}[h^{1/4}\cY,h^{-1/4}\nabla_\cY]\,.
\] 
We expand the right hand side as a formal series of operators $\sum_{j} h^{j/4}\cL_j[Y,\nabla_\cY]$ defined on $\R^2$. We have
\[
   \sum_{j} h^{j/4}\cL_j[Y,\nabla_\cY] =
   h^{-3/2} \Big(-ih^{3/4}\nabla_\cY+\sum_{j\ge0} h^{(3+\ell)/4}\AAA_j(\cY)\Big)^2.
\]
Hence the series $\sum_{j} h^{j/4}\cL_j$ starts with $\cL_0$ given by
\[
   \cL_0 = \big(-i\nabla_\cY+\AAA_0(\cY)\big)^2
\]
and the other terms $\cL_j$ for $j\ge1$ are partial differential operators of degree $1$ with polynomial coefficients. The main term $\cL_0$ is isospectral to $(-i\nabla+\Xi(\cx_0)\AAA_\varepsilon(\cx_0))^2$, see \eqref{eq:OpSigma}, and its eigenvalues are given by the $\Lambda^{\cx_0}_n$ ($n\in\N^*$), see \eqref{eq:Lam1}.

Choose a normalized eigenpair of $\cL_0$, which we denote by $(\ell_0,\Psi_0)$. We look for  
\[
   \ell_j\in\R \quad\mbox{and}\quad \Psi_j\in\Dom(\cL_0),\quad j=1,2,\ldots
\]
solving 
$
   \big(\sum_{j} h^{j/4}\cL_j\big)\big(\sum_k h^{k/4}\Psi_k\big) =
   \big(\sum_{j} h^{j/4}\ell_j\big)\big(\sum_k h^{k/4}\Psi_k\big)
$
in the sense of formal series,
i.e., solving the sequence of equations, 
\[
   \sum_{j=0}^m \cL_j \Psi_{m-j} =
   \sum_{j=0}^m \ell_j \Psi_{m-j}\quad m=1,2,\ldots.
\]
We write the first equation (for $m=1$) as
\[
   (\cL_0 -\ell_0)\Psi_1 = \ell_1\Psi_0 - \cL_1\Psi_0\,,
\]
and the next ones as
\[
   (\cL_0 -\ell_0)\Psi_m = \ell_m\Psi_0 - \cL_m\Psi_0 + \sum_{j=1}^{m-1} (\ell_j-\cL_j) \Psi_{m-j}\,.
\]
If $\ell_0$ is a {\em simple} eigenvalue of $\cL_0$, the solution of such a sequence of equations is classical, resulting from the Fredholm alternative for the self-adjoint operator $\cL_0$. For instance, we get 
\begin{equation}
\label{eq:Fred}
   \ell_1 = \langle \cL_1\Psi_0,\Psi_0\rangle
\end{equation}
if $m=1$. If $\ell_0$ is a {\em multiple} eigenvalue of $\cL_0$, we cannot choose a priori an associated eigenvector, but have to work in the whole associated eigenspace $E_{\ell_0}$. Then identity \eqref{eq:Fred} is replaced by an eigen-equation for a finite dimensional hermitian matrix acting on the eigenspace $E_{\ell_0}$. The process can be pursed as well, see \cite{DDFR99} for details on this procedure. The terms $\Psi_m$ belong to the domain of $\cL_0$ and have, furthermore, exponential decay. Setting for $m\ge1$
\[
   \ell^{[m]}(h) = h^{3/2}\sum_{j=0}^m h^{j/4}\ell_j \quad\mbox{and}\quad
   \psi_h^{[m]}(\cy) = \chi_{\cx_0}(\cy) \sum_{j=0}^m h^{j/4}\Psi_j(\cy/h^{1/4})
\]
we have constructed a quasimode for $\Op$ to the order $h^{3/2 + (m+1)/4}$, i.e.
\[
   \Op(\psi_h^{[m]}) = \ell^{[m]}(h) \psi_h^{[m]} + \rho^{[m+1]}_h\quad\mbox{with}\quad
   \|\rho^{[m+1]}_h\| \le C_m h^{3/2 + (m+1)/4} \|\psi^{[m]}_h\|.
\]
Combining this with \eqref{eq:asy1}, we deduce by the spectral theorem that there holds

\begin{thm}\label{thm:asy}
Under Assumption \textup{\ref{hypcadre2}}, for all integers $N\ge1$ and $M\ge0$, there exist coefficients
\[
   \Lambda^{\BB}_{n,m}\in\R^+\quad\mbox{for}\quad 1\le n\le N,\:\:0\le m\le M,
   \quad\mbox{with}\quad \Lambda^{\BB}_{n,0}=\Lambda^{\BB}_{n}
\]
and constants $C_{N,M}>0$ and $h_0>0$ such that, for all $h\in (0,h_0)$ and all $n=1,\ldots,N$
\[
   \Big|\lambda_n(h)-h^{3/2}\sum_{m=0}^Mh^{m/4}\Lambda_{n,m}^{\BB}\Big| \le C_N h^{3/2+(M+1)/4}\,.
\]
\end{thm}
\noindent
Of course, Theorem \ref{thm:asy1} is a particular case of the above statement if we choose $M=0$.

\section{Small angle limit}\label{PSA}
The first part of this section is devoted to theoretical results on the band function $(\alpha,\xi)\mapsto\varrho_1(\alpha,\xi)$ and to their numerical illustration. In the second part, we rely on these results to prove the convergence of eigenvalues of $\Cr$ in the small angle limit and present the computations of their first eigenstates for a set of small values of $\varepsilon$.

\subsection{Operator symbol and band function}
Here we study the behavior of the first eigenvalue $\varrho_1(\alpha,\xi)$ of the operator symbol $\sX_{\alpha, \xi}$, for $(\alpha,\xi)\in\R^2$ 
\[
   \sX_{\alpha, \xi} = D_t^2+\Big(\xi+ \alpha^2 t-\frac{t^3}{3}\Big)^2\,,
\]
acting on $L^2(\R)$, see \eqref{eq:Xalxi}. 

\subsubsection{Preliminaries}
Let us introduce the potential $\sV_{\alpha,\xi}$ of $\sX_{\alpha, \xi}$ and its generating polynomial $\sP_{\alpha,\xi}$:
\[
   \sP_{\alpha,\xi}(t) = \xi+ \alpha^2 t-\frac{t^3}{3} \quad\mbox{and}\quad
   \sV_{\alpha,\xi}(t) = \Big(\sP_{\alpha,\xi}(t)\Big)^2.
\]

\[
   \sV_{\alpha,\xi}(t) = \xi^2 + 2\xi\alpha^2t - 2\xi \frac{t^3}{3} + \alpha^4 t^2 
   - 2\alpha^2\frac{t^4}{3} + \frac{t^6}{9}
\]

The potential $\sV_{\alpha,\xi}$ depends smoothly on the parameters $(\alpha,\xi)$ and is confining for each value of $(\alpha,\xi)$. So there holds
\begin{prop}
For all $(\alpha,\xi)\in \R^2$, the operator $\sX_{\alpha,\xi}$ has a compact resolvent and
the family $\left(\sX_{\alpha, \xi}\right)$ with $(\alpha, \xi) \in \R^2$ is analytic (of type ($B$) according to Kato theory, see \cite{KATO}). 
\end{prop}

Since the $\sX_{\alpha, \xi}$ are Sturm-Liouville operators, we obtain

\begin{cor}
For all $(\alpha,\xi)$, the eigenvalue $\varrho_1(\alpha,\xi)$ is simple and depends analytically of $\alpha$ and $\xi$. The associated eigenfunctions do not vanish and the unique normalized and positive eigenfunction $u_{\alpha,\xi}$ associated with $\varrho_1(\alpha,\xi)$ depends analytically of $(\alpha,\xi)$.
\end{cor}

As a consequence we have the following \enquote{Feynman-Hellmann} formulas.

\begin{cor}\label{idFHmin}
For all $(\alpha,\xi)$ we have the following identities 
\begin{equation*}
(\partial_{\alpha}
\varrho_{1})(\alpha,\xi)=4\alpha\displaystyle{\int_{\R}\left(\xi+\alpha^2t-\frac{t^3}{3}\right)t u_{\alpha,\xi}^2(t)\,\rd t},
\end{equation*}
\begin{equation*}(\partial_{\xi}
\varrho_{1})(\alpha,\xi)=2\displaystyle{\int_{\R}\left(\xi+\alpha^2t-\frac{t^3}{3}\right)u_{\alpha,\xi}^2(t)\,\rd t}.
\end{equation*}
\end{cor}

The potential $\sV_{\alpha,\xi}$ has the following obvious symmetry properties: $\sV_{\alpha,\xi}(t)=\sV_{-\alpha,\xi}(t)$ and $\sV_{\alpha,-\xi}(t)=\sV_{\alpha,\xi}(-t)$. Hence the band function $\varrho_{1}(\alpha,\xi)$ is even with respect to each of the two variables $\alpha$ and $\xi$, so its analysis can be restricted to the first quadrant $\{\alpha\ge0, \xi\ge 0\}$.
The following lemma gives an expression of the roots of the generating polynomial $\sP_{\alpha,\xi}$, depending on the sign of its discriminant. 

\begin{lem}\label{racinesP}
For all $\alpha\ge0$ and all $\xi\ge0$, denote by $t_k(\alpha,\xi)$ the three roots of the polynomial $\sP_{\alpha,\xi}$, agreeing that
\[
   \Re t_1(\alpha,\xi) \le \Re t_2(\alpha,\xi) \le \Re t_3(\alpha,\xi).
\] 
Then, if $(\alpha,\xi)\neq(0,0)$, $t_3(\alpha,\xi)$ is a positive simple real root. More precisely we have
\begin{subequations}
\begin{itemize}
\item For $\xi<\frac{2}{3}\alpha^3$, the polynomial $\sP_{\alpha,\xi}$ admits three distinct real roots (that we denote $t_{1}(\alpha,\xi)<t_{2}(\alpha,\xi)<t_{3}(\alpha,\xi)$) given (for all $k\in\{1,2,3\}$) by
\begin{equation}
\label{eq:ta}
   t_{k}(\alpha,\xi)=j^{3-k}\sqrt[3]{\frac{1}{2}\left(3\xi+i\sqrt{4\alpha^6-9\xi^2}\right)}
   +j^{k-3}\sqrt[3]{\frac{1}{2}\left(3\xi-i\sqrt{4\alpha^6-9\xi^2}\right)},
\end{equation}
where $j$ is the complex number defined by $j=e^{2i\pi/3}=-\frac{1}{2}+i\frac{\sqrt{3}}{2}$.
\item For $\xi=\frac{2}{3}\alpha^3$, the polynomial $\sP_{\alpha,\xi}$ admits a simple real root and a double real root respectively given by
\begin{equation}
\label{eq:tb}
   t_{3}(\alpha,\xi)=2\alpha \ \text{and} \ t_{2}(\alpha,\xi)=t_{1}(\alpha,\xi)=-\alpha .
\end{equation}
\item For $\xi>\frac{2}{3}\alpha^3$, the polynomial $\sP_{\alpha,\xi}$ admits a unique real root given by 
\begin{equation}
\label{eq:tc}
   t_{3}(\alpha,\xi)=\sqrt[3]{\frac{1}{2}\left(3\xi+\sqrt{-4\alpha^6+9\xi^2}\right)}
   +\sqrt[3]{\frac{1}{2}\left(3\xi-\sqrt{-4\alpha^6+9\xi^2}\right)}.
\end{equation}
\end{itemize}
\end{subequations}
\end{lem}
The next result shows that the minimum cannot be reached on the set $\{\xi\ge\frac{2}{3}\alpha^3>0\}$. Its proof uses Corollary \ref{idFHmin} and the fact that, if $\xi\ge\frac{2}{3}\alpha^3>0$, the polynomial $t\mapsto(t-t_3)\sP_{\alpha,\xi}(t)$ is negative.
\begin{prop}
\label{prop:critic}
For all $(\alpha,\xi)$ such that $\xi\ge\frac{2}{3}\alpha^3>0$, we have
$$(\partial_{\alpha}
\varrho_{1})(\alpha,\xi)-2\alpha t_{3}(\alpha,\xi)(\partial_{\xi}
\varrho_{1})(\alpha,\xi)<0.$$
In particular, there is no critical point on the set $\{\xi\ge\frac{2}{3}\alpha^3>0\}$.
\end{prop}

\subsubsection{Behavior of the band function at infinity}
Now, the remaining part of this section is devoted to prove that $\varrho_1(\alpha,\xi)$ tends to infinity as $|\alpha|+|\xi|$ tends to infinity, namely
\begin{equation}
\label{eq:varrhoinfty}
   \lim_{|\alpha|+|\xi|\to\infty} \varrho_1(\alpha,\xi) = \infty\,.
\end{equation}
Note that Proposition \ref{prop:critic} combined with \eqref{eq:varrhoinfty} implies Theorem \ref{thmprincipalsymboleIntro}.

To prove \eqref{eq:varrhoinfty}, we split (for each $R>1$) the region
\[
   \sA_R := \{(\alpha,\xi)\in\R^2,\quad \alpha\ge0,\ \ \xi\ge0,\ \ \mbox{and}\ \ \alpha+\xi>R\}
\]
into the three subregions
\begin{subequations}
\label{eq:subA}
\begin{align}
   \sA^\circ_R &:= \sA_R \cap \{(\alpha,\xi)\in\R^2,\quad \alpha\in[0,1] \quad \xi>\tfrac{2}{3}\alpha^3\}\\
   \sA^\sharp_R &:= \sA_R \cap \{(\alpha,\xi)\in\R^2,\quad \alpha\in[1,\infty)\quad \xi>\tfrac{2}{3}\alpha^3\}\\
   \sA^\flat_R &:= \sA_R \cap \{(\alpha,\xi)\in\R^2,\quad \alpha\in[1,\infty)\quad \xi\le\tfrac{2}{3}\alpha^3\}
\end{align}
\end{subequations}
and are going to prove the next lemma.

\begin{lem}
\label{lem:A}
We denote $\sQ_{\alpha,\xi}$ the quadratic form associated with the operator $\sX_{\alpha,\xi}$.
There exists constants $R>1$ and $B>0$ such that for all $\psi\in\mathcal{C}^\infty_0(\R)$ the following lower bounds hold
\begin{subequations}
\label{eq:A}
\begin{align}
\label{eq:Acirc}
   \sQ_{\alpha,\xi}(\psi) &\ge B\,\xi^{2/3}\|\psi\|_{L^2(\R)}^2,\quad \forall(\alpha,\xi)\in\sA^\circ_R \\
\label{eq:Asharp}
   \sQ_{\alpha,\xi}(\psi) &\ge B\,\xi^{2/9}\|\psi\|_{L^2(\R)}^2,\quad \forall(\alpha,\xi)\in\sA^\sharp_R \\
\label{eq:Aflat}
   \sQ_{\alpha,\xi}(\psi) &\ge B\,\alpha^{2/3}\|\psi\|_{L^2(\R)}^2,\quad \forall(\alpha,\xi)\in\sA^\flat_R
\end{align}
\end{subequations}
\end{lem}

We can see that 
\[
   \xi^{2/9}> \tfrac12\xi^{2/9}+\tfrac13 \alpha^{2/3} \ \ \mbox{on}\ \ \sA^\circ_R\cup\sA^\sharp_R
   \quad\mbox{and}\quad
   \alpha^{2/3}> \tfrac12(\alpha^{2/3}+ \xi^{2/9}) \ \ \mbox{on}\ \ \sA^\flat_R\,.
\]
Therefore the bounds \eqref{eq:A} imply \eqref{eq:varrhoinfty}.

\begin{proof}[Proof of \eqref{eq:Acirc}]
We recall that, on the set $\{\xi>\frac{2}{3}\alpha^3\ge 0\}$, the polynomial $\sP_{\alpha,\xi}$ admits a unique real root, denoted $t_{3}(\alpha,\xi)$. Using that $\sP_{\alpha,\xi}(t_3(\alpha,\xi))=0$, we immediately check that we have the  factorization
\begin{equation}\label{fact}
   \sP_{\alpha,\xi}(t) =
   -\sN_{\alpha,\xi}(t)\big(t-t_3(\alpha,\xi)\big)
   \quad\mbox{with}\quad
   \sN_{\alpha,\xi}(t)=\frac{1}{3}t^2+\frac{t_3(\alpha,\xi)}{3}t+\frac{\xi}{t_3(\alpha,\xi)}\,.
\end{equation}
The factor $\sN_{\alpha,\xi}(t)$ is positive for all $t$ when $\xi>\frac{2}{3}\alpha^3$ and
we have $\partial_t\sN_{\alpha,\xi}(t)=0$ for $t=-\frac{t_3(\alpha,\xi)}{2}$. Therefore, for all $\xi> \frac{2}{3}\alpha^3$, we have the lower bound
\begin{equation*}
   \sN_{\alpha,\xi}(t)\ge -\frac{t_3(\alpha,\xi)^2}{12}+\frac{\xi}{t_3(\alpha,\xi)}.
\end{equation*}
Then, the quotient $\sQ_{\alpha,\xi}(\psi) / \|\psi\|_{L^2(\R)}^2$ is bounded from below by the ground state energy of the operator
$$
   D_t^2+\left(-\frac{t_3(\alpha,\xi)^2}{12}+\frac{\xi}{t_3(\alpha,\xi)}\right)^2(t-t_3(\alpha,\xi))^2
   \quad\mbox{on}\quad\R.
$$
By considering the expression of $t_3(\alpha,\xi)$ given in \eqref{eq:tc}, we get, uniformly in $\alpha\in [0,1]$,
\begin{equation*}
   t_3(\alpha,\xi) = (3\xi)^{1/3} + \underset{\xi\rightarrow+\infty}{\mathcal{O}(\xi^{-1/3})}.
\end{equation*}
Hence, there exist constants $B>0$ and $K>0$ such that for all $\xi>K$ (with $\xi>\frac{2}{3}\alpha^3$ and $0\le \alpha\le 1$)
\begin{equation*}
   \left(-\frac{t_3(\alpha,\xi)^2}{12}+\frac{\xi}{t_3(\alpha,\xi)}\right)\ge B\,\xi^{2/3}.
\end{equation*}
For all $\xi>K$, $\frac{\sQ_{\alpha,\xi}(\psi)}{\|\psi\|_{L^2(\R)}^2}$ is bounded from below by the ground state energy of
$$
   D_t^2+B^2\xi^{4/3}(t-t_3(\alpha,\xi))^2 \quad\mbox{on}\quad\R.
$$
By translation and homogeneity, we get (using the harmonic oscillator):
\begin{equation}\label{eq:in2bis}
\sQ_{\alpha,\xi}(\psi)\ge B\,\xi^{2/3}\|\psi\|_{L^2(\R)}^2.
\end{equation}
This concludes the proof of the estimate \eqref{eq:Acirc}.
\end{proof}

\begin{proof}[Preliminaries for the proof of \eqref{eq:Asharp} and \eqref{eq:Aflat}]
For the proof of estimates \eqref{eq:Asharp} and \eqref{eq:Aflat}, we use a quadratic partition of unity $(\chi_1,\chi_2)$ on $\R$ in order to isolate the root $t_3(\alpha,\xi)$ from the other two roots of $\sP_{\alpha,\xi}$. For this we choose two real numbers $\alpha_{-}$ and $\alpha_{+}$ such that
\[
   \alpha_{-}=\gamma_{-}\alpha,\quad \alpha_{+}=\gamma_{+}\alpha\quad\mbox{with}\quad
   0\le\gamma_{-}<\gamma_{+}
\]
and we take the two functions $\chi_1$ and $\chi_2$ such that $\chi_{1}^2+\chi_{2}^2=1$ on $\R$ and
\begin{equation*}
\begin{split}
   \chi_1(t) =\left\{
   \begin{array}{ccc} 1 & \text{on} & (-\infty,\alpha_{-}]
   \\ 0 & \text{on} & [\alpha_{+},+\infty)
   \end{array}\right. \quad 
   \text{and} \quad 
   \chi_2(t) =\left\{
   \begin{array}{ccc} 0 & \text{on} & (-\infty,\alpha_{-}]
   \\ 1 & \text{on} & [\alpha_{+},+\infty)
   \end{array}\right.
\end{split}
\end{equation*}
with the control of their derivatives
\begin{equation*}
   \sup_{t\in\R}|\chi_{j}'(t)|\le \frac{\Cims}{\alpha_{+}-\alpha_{-}}, \quad \text{for all} \ j\in \{1,2\}.
\end{equation*}
The localization formula (see \eqref{formule de localisation}) gives, for all function $\psi$ in the form domain,
\begin{equation}\label{IMSmin}
   \sQ_{\alpha,\xi}(\psi)=
   \sQ_{\alpha,\xi}(\chi_1\psi)+\sQ_{\alpha,\xi}(\chi_2\psi)
   -\| \psi\chi_1'\|^2_{L^2(\R)}-\| \psi\chi_2'\|^2_{L^2(\R)}.
\end{equation}
Whence
\begin{equation}\label{Restemin}
   \sQ_{\alpha,\xi}(\psi) \ge
   \sQ_{\alpha,\xi}(\chi_1\psi)+\sQ_{\alpha,\xi}(\chi_2\psi)
   - \frac{\Cims^2}{(\alpha_{+}-\alpha_{-})^2}\, \| \psi\|^2_{L^2(\R)}\,.
\end{equation}
We let $\Omega^{(1)}=(-\infty,\alpha_{+})$ and $\Omega^{(2)}=(\alpha_{-},+\infty)$ and denote $\psi_1=\chi_1\psi$ and $\psi_2=\chi_2\psi$. We will work out a lower bound of the quadratic form on each of these subdomains. 

\smallskip
On both subdomains, using the factorization \eqref{fact}, we start from the expression of $\sX_{\alpha,\xi}$ as
\begin{equation}
\label{eq:sX}
   \sX_{\alpha,\xi} = D_{t}^2 + \sV_{\alpha,\xi} \quad\mbox{with}\quad
   \sV_{\alpha,\xi} = \big(\sN_{\alpha,\xi}(t)\big)^2\, \big(t-t_3(\alpha,\xi)\big)^2.
\end{equation}
On $\Omega^{(1)}$, we bound from below the potential $\sV_{\alpha,\xi}$ as
\begin{equation}
\label{eq:B1}
   \sV_{\alpha,\xi} \ge  B_1(\alpha,\xi)\,\big(\sN_{\alpha,\xi}(t)\big)^2
   \quad\mbox{with}\quad
   B_1(\alpha,\xi) = \min_{t\le\alpha_{+}} \big(t-t_3(\alpha,\xi)\big)^2.
\end{equation}
Therefore, the quotient $\sQ_{\alpha,\xi}(\psi_1) / \|\psi_1\|^2_{L^2(\R)}$ is bounded from below by the ground state energy of
\begin{equation}
\label{eq:Op1}
   D_{t}^2 + B_1(\alpha,\xi)\,\big(\sN_{\alpha,\xi}(t)\big)^2 \quad\mbox{on}\quad \R.
\end{equation}
Using the canonical form of the factor $\sN_{\alpha,\xi}$
\begin{equation}
\label{eq:can}
   \sN_{\alpha,\xi}(t)=\frac{1}{3}\Big(t+t_3(\alpha,\xi)\Big)^2
   -\frac{t_3(\alpha,\xi)^2}{12}+\frac{\xi}{t_3(\alpha,\xi)}\,.
\end{equation}
we find by translation and scaling that the operator \eqref{eq:Op1} is isospectral to the operator
\begin{equation}
\label{eq:Op1b}
   (\tfrac{4}{9})^{1/3}B_1(\alpha,\xi)^{1/3} \left(D_{\tau}^2 + \Big(\frac{\tau^2}{2} -\eta\Big)^2\right)^2 \quad\mbox{on}\quad \R,
\end{equation}
for a suitable real number $\eta$. We know from \cite{H10} that the ground state energy $\gamma(\eta)$ of the operator $D_{\tau}^2 + \big(\frac{\tau^2}{2} -\eta\big)^2$ as a function of $\eta\in\R$ reaches its minimum (for a positive value $\eta_0$ of $\eta$). Therefore this minimum is positive. We denote it by $\mathrm{M}_0$. Finally we bound $\sQ_{\alpha,\xi}(\psi_1)$ from below as follows
\begin{equation}
\label{eq:psi1}
   \sQ_{\alpha,\xi}(\psi_1) \ge (\tfrac{4}{9})^{1/3}B_1(\alpha,\xi)^{1/3}\mathrm{M}_0
   \|\psi_1\|^2_{L^2(\R)}.
\end{equation}

\smallskip
On $\Omega^{(2)}$, we swap the roles of the two factors in $\sP_{\alpha,\xi}$ and obtain the lower bound:
\begin{equation}
\label{eq:B2}
   \sV_{\alpha,\xi} \ge  B_2(\alpha,\xi)\,\big(t-t_3(\alpha,\xi)\big)^2
   \quad\mbox{with}\quad
   B_2(\alpha,\xi) = \min_{t\ge\alpha_{-}} \big(\sN_{\alpha,\xi}(t)\big)^2.
\end{equation}
By translation and homogeneity, we get via the harmonic oscillator:
\begin{equation}
\label{eq:psi2}
   \sQ_{\alpha,\xi}(\psi_2) \ge B_2(\alpha,\xi)^{1/2}
   \|\psi_2\|^2_{L^2(\R)}\,.
\end{equation}

\smallskip
Finally, combining \eqref{eq:psi1} and \eqref{eq:psi2} with \eqref{Restemin} we find
\begin{equation}
\label{eq:psi}
   \sQ_{\alpha,\xi}(\psi) \ge \left(\min\Big\{(\tfrac{4}{9})^{1/3}B_1(\alpha,\xi)^{1/3}\mathrm{M}_0,\:
   B_2(\alpha,\xi)^{1/2}\Big\} - \frac{\Cims^2}{(\alpha_{+}-\alpha_{-})^2}\right) \|\psi\|^2_{L^2(\R)}\,.
\end{equation}
It remains to choose $\alpha_{+}$ and $\alpha_{-}$ so that we can find suitable lower bounds for the constants $B_1(\alpha,\xi)$ and $B_2(\alpha,\xi)$. This will be done finding upper and lower bounds for $t_3(\alpha,\xi)$.
\end{proof}

\begin{proof}[Proof of \eqref{eq:Asharp}]
In the region $\sA^\sharp$, we have $\xi>\frac{2}{3}\alpha^3$ (and $\alpha>1$). With the aim of finding bounds for $t_3(\alpha,\xi)$ we calculate the derivative of expression \eqref{eq:tc} with respect to $\alpha$:
\begin{multline*}
    \partial_\alpha t_{3}(\alpha,\xi) = -\frac{4\alpha^5}{\sqrt{-4\alpha^6+9\xi^2}}
    \left(\frac{1}{2}\left(3\xi+\sqrt{-4\alpha^6+9\xi^2}\right)\right)^{-2/3} \\
   +\frac{4\alpha^5}{\sqrt{-4\alpha^6+9\xi^2}}
    \left(\frac{1}{2}\left(3\xi-\sqrt{-4\alpha^6+9\xi^2}\right)\right)^{-2/3} .
\end{multline*}
We can see that the modulus of the second term is larger than the modulus of the first term. Hence the positivity of the derivative $\partial_\alpha t_{3}(\alpha,\xi)$. Therefore
\[
   t_{3}(0,\xi) \le t_{3}(\alpha,\xi) \le t_{3}\big((\tfrac32\xi)^{1/3},\xi\big),\quad\forall\xi>0.
\]
With \eqref{eq:tc} (and using again that $\xi>\frac{2}{3}\alpha^3$), 
we deduce
\[
   2^{1/3} \alpha <
   (3\xi)^{1/3} \le t_{3}(\alpha,\xi) \le 2^{2/3} (3\xi)^{1/3}.
\]
We choose 
\[
   \alpha_{-} = 0\quad\mbox{and}\quad \alpha_{+} = \alpha.
\]
Thus, we get for the constants $B_1(\alpha,\xi)$ and $B_2(\alpha,\xi)$ appearing in \eqref{eq:B1} and \eqref{eq:B2}:
\[
   B_1(\alpha,\xi) = \min_{t\le\alpha_{+}} \big(t-t_3(\alpha,\xi)\big)^2 \ge 
   (\alpha_{+}-(3\xi)^{1/3}\big)^2 \ge (3\xi)^{2/3} (1-2^{-1/3})^2
\]
and
\[
\begin{aligned}
   B_2(\alpha,\xi) &= \min_{t\ge\alpha_{-}} \big(\sN_{\alpha,\xi}(t)\big)^2 \\ &=
   \min_{t\ge\alpha_{-}} \left(\frac{1}{3}t^2+\frac{t_3(\alpha,\xi)}{3}t+\frac{\xi}{t_3(\alpha,\xi)}\right)^2 
   \ge \left(\frac{\xi}{t_3(\alpha,\xi)}\right)^2 
   \ge 2^{-4/3}3^{-2/3} \xi^{4/3}.
\end{aligned}
\]
Then \eqref{eq:psi} yields
\[
   \sQ_{\alpha,\xi}(\psi) \ge \left(\min\big\{C_1\xi^{2/9},\:
   C_2\xi^{2/3}\big\} - \Cims^2 \alpha^{-2}\right) \|\psi\|^2_{L^2(\R)}\,.
\]
Since $\alpha\ge1$, this clearly implies \eqref{eq:Asharp} if $\xi$ is large enough.
\end{proof}

\begin{proof}[ Proof of \eqref{eq:Aflat}]
In the region $\sA^\flat$, $\frac23\alpha^3\ge\xi$ and the polynomial $\sP_{\alpha,\xi}$ has three real roots. We note that $\sP_{\alpha,\xi}(t)\to-\infty$ as $t\to+\infty$ and $\sP_{\alpha,\xi}(t)\to+\infty$ as $t\to-\infty$. We check that
\begin{multline*}
   \sP_{\alpha,\xi}(-\sqrt3\alpha) = \xi  > 0,\qquad
   \sP_{\alpha,\xi}(-\alpha) = \xi - \tfrac23\alpha^3 \le 0,\\
   \sP_{\alpha,\xi}(\sqrt3\alpha) = \xi  > 0,\qquad
   \sP_{\alpha,\xi}(2\alpha) = \xi - \tfrac23\alpha^3 \le 0.
\end{multline*}
This implies that
\[
   -\sqrt3\alpha < t_1(\alpha,\xi) \le -\alpha \le t_2(\alpha,\xi) < \sqrt3\alpha < t_3(\alpha,\xi) \le 2\alpha.
\]
Now we choose 
\[
   \alpha_{-} = \tfrac12\alpha \quad\mbox{and}\quad \alpha_{+} = \alpha,
\]
and we get for the constants $B_1(\alpha,\xi)$ and $B_2(\alpha,\xi)$:
\[
   B_1(\alpha,\xi) = \min_{t\le\alpha_{+}} \big(t-t_3(\alpha,\xi)\big)^2 \ge 
   (\alpha_{+}-\sqrt3\alpha\big)^2 = \alpha^2 (1-\sqrt3)^2
\]
and
\[
   B_2(\alpha,\xi) = 
   \min_{t\ge\alpha_{-}} \left(\frac{1}{3}t^2+\frac{t_3(\alpha,\xi)}{3}t+\frac{\xi}{t_3(\alpha,\xi)}\right)^2 
   \ge \left(\frac{\alpha^2}{12} \right)^2 = \frac{1}{144} \alpha^4.
\]
Then \eqref{eq:psi} yields
\[
   \sQ_{\alpha,\xi}(\psi) \ge \left(\min\big\{C_1\alpha^{2/3},\:
   C_2\alpha^{2}\big\} - \Cims^2 \alpha^{-2}\right) \|\psi\|^2_{L^2(\R)}\,,
\]
which implies \eqref{eq:Aflat}.
\end{proof}

The proof of Lemma \ref{lem:A} is now achieved, hence Theorem \ref{thmprincipalsymboleIntro} is proved.

\subsubsection{Numerical simulations}
We have computed an approximation of the band function $\varrho_1$ on a grid $\Sigma$ of values of $(\alpha,\xi)$ covering the square $(-2,2)^2$. The grid points of $\Sigma$ are $(\alpha_k, \xi_l)$ with $\alpha_k=-2+{k}/{100}$ and  $\xi_l=-2+{l}/{100}$, for $k,l \in \{0,\ldots, 400\}$. In Figure \ref{F1}, we plot the level lines of $\varrho_1$ above the grid $\Sigma$ and in Figure \ref{F2}, we plot the same band function restricted on the axis $\xi=0$.

\begin{figure}[ht]
\begin{center}
    \includegraphics[width=0.9\textwidth]{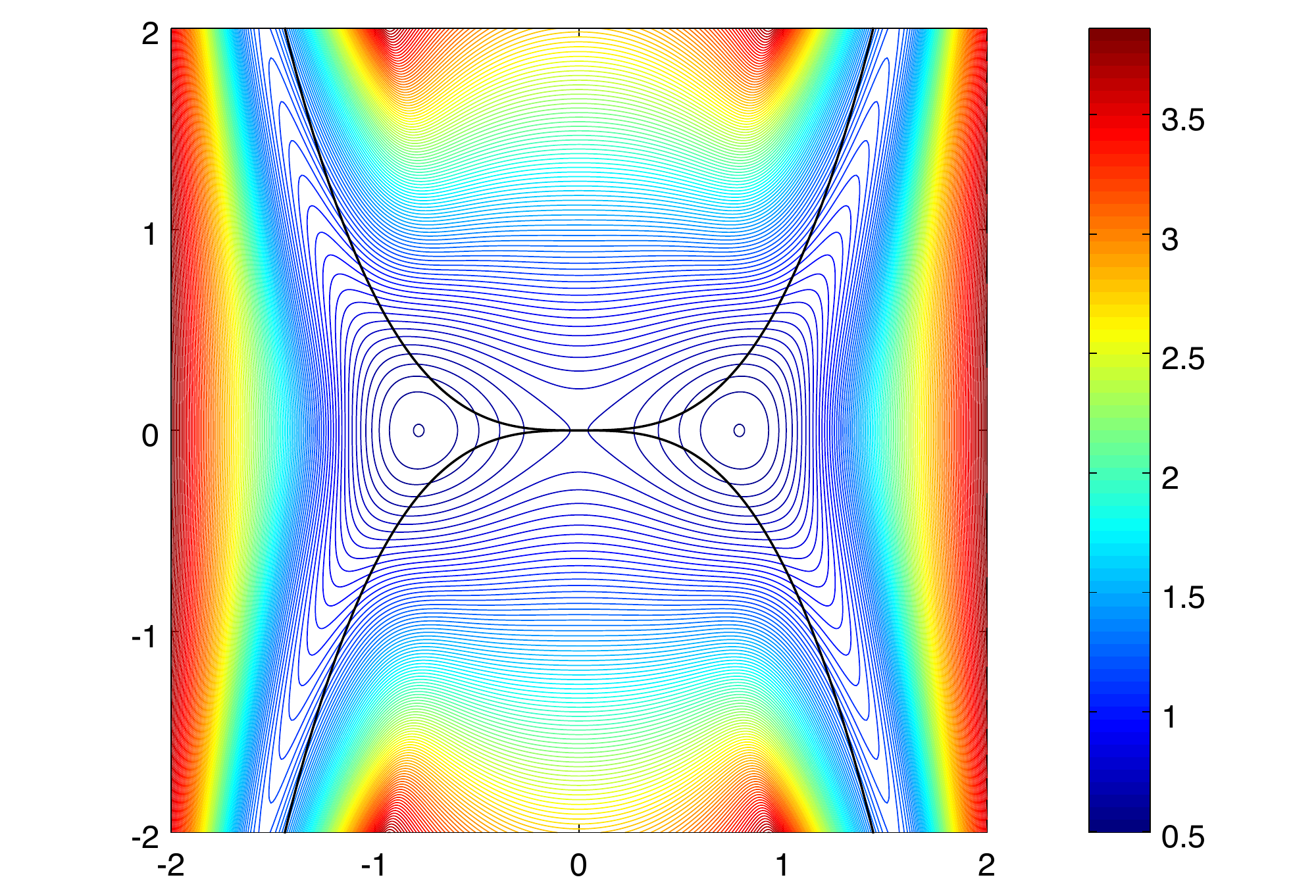}
    
    \includegraphics[width=0.9\textwidth]{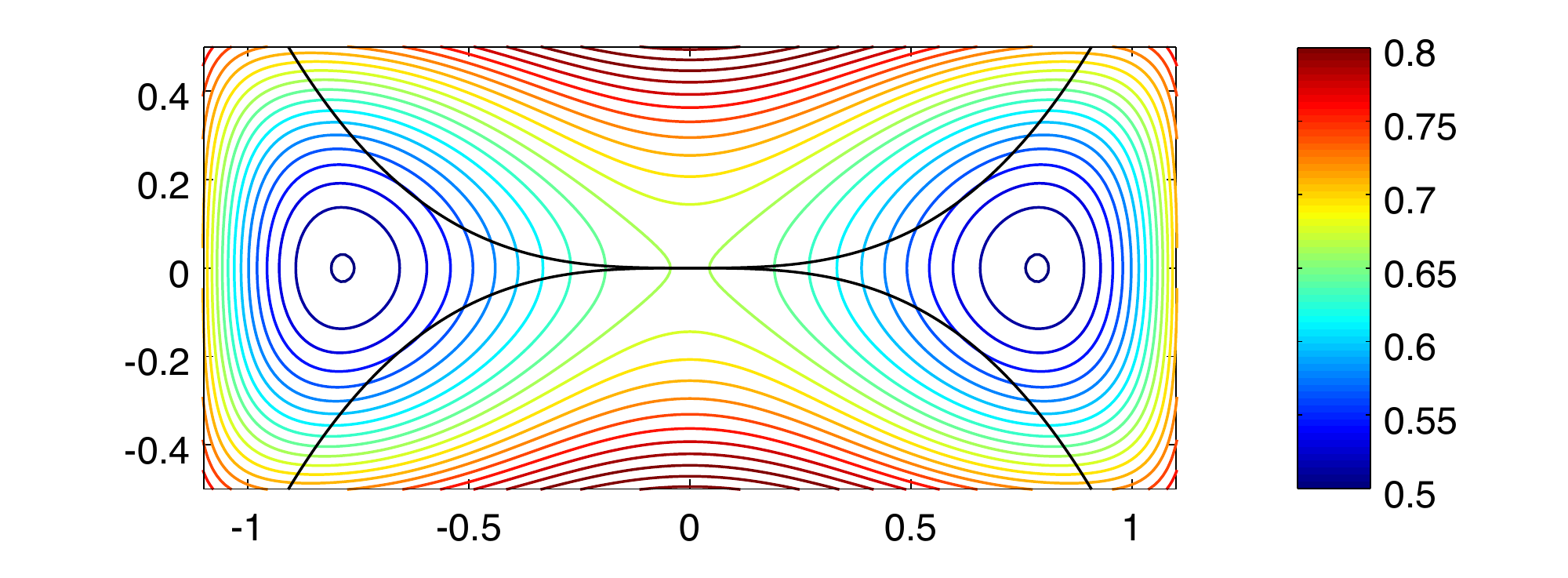}
    \end{center}
    \caption{Contour lines for the numerical values of $\varrho_1(\alpha, \xi)$ on the grid $\Sigma$, with a zoom in a neighborhood of the two points where the minimum is attained.
The curves $\{\alpha\mapsto\xi=\pm\frac{2}{3}\alpha^3\}$ are also represented.
    }
    \label{F1}
\end{figure}

\begin{figure}[ht]
\begin{center}
    \includegraphics[width=0.7\textwidth]{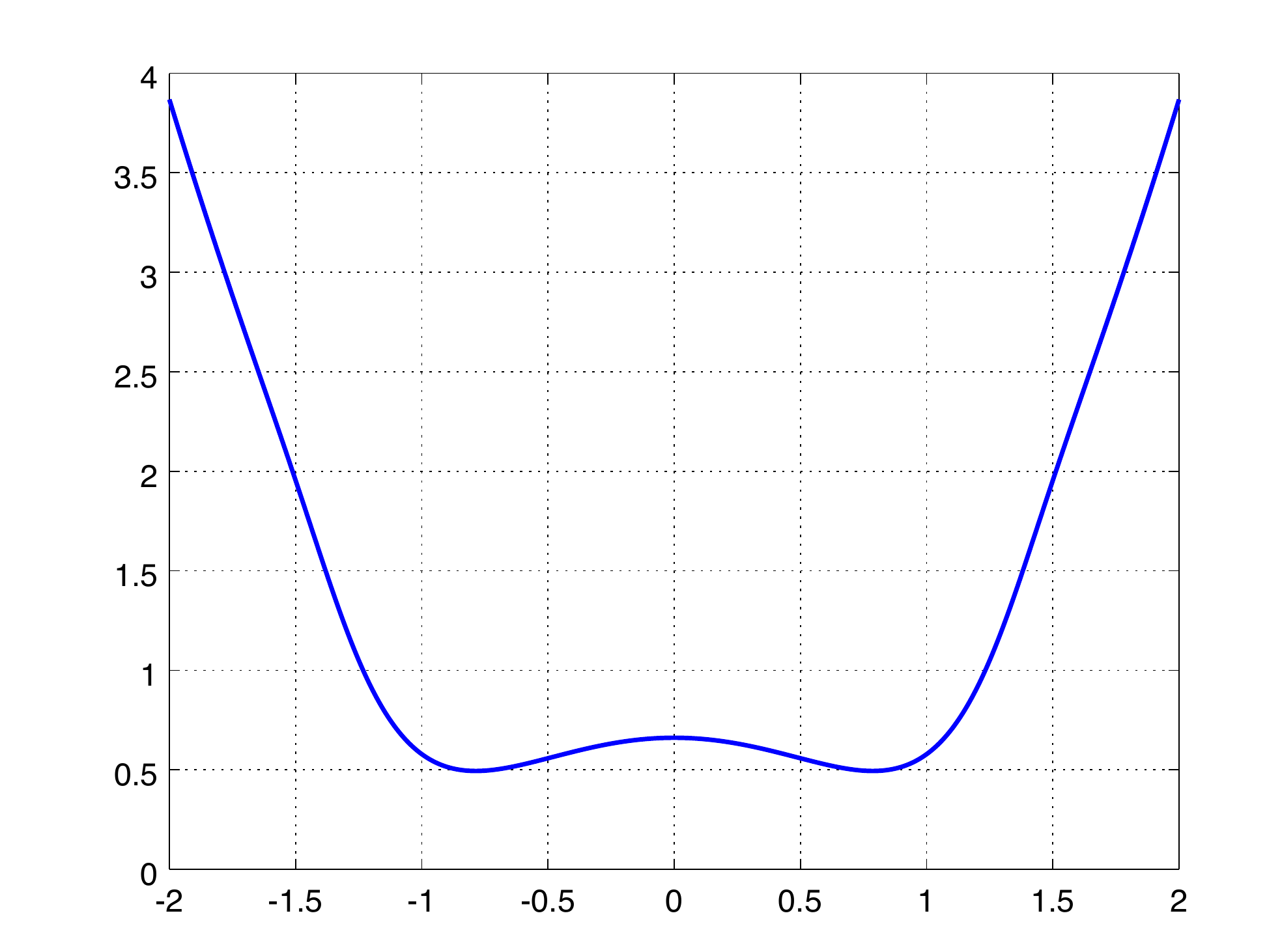}
    \end{center}
    \caption{Numerical values for $\varrho_1(\alpha, \xi)$ on the axis $\xi=0$.
    }
    \label{F2}
\end{figure}

The computations for Figures \ref{F1} and \ref{F2} are performed by the finite element method\footnote{All our computations are preformed with the FEM library \href{https://uma.ensta-paristech.fr/soft/XLiFE++/}{XLiFE++} under a \href{https://www.gnu.org/licenses/gpl-3.0.en.html}{GNU GPL licence}.} and are based on a Galerkin projection on the interval $(-5,5)\ni t$ with natural boundary conditions at the ends, discretized by 10 elements with polynomial degree 10. With this number of elements the degree $10$ saturates the double precision, see Table \ref{T1}. Enlarging the domain yields numbers with the same first $13$ digits, which proves that $(-5,5)$ is large enough to capture the numerical support of the first eigenvectors when $(\alpha,\xi)$ belongs to $[-2,2]^2$.

\begin{table}[b]
\small
\begin{tabular}{|c|c|c|c|}
\hline
$Q$ & $\varrho_1(0,0)$ & $\alpha_0$ & $\varrho_1(\alpha_0,0)$
$\vphantom{L^{\displaystyle 2}_{\displaystyle 2_2}}$\\
\hline
1 & 0.716813090776313 &   0.794 &  0.549407920248045\\
2 & 0.665333352584016 &   0.790 &  0.495300498319300\\
3 & 0.660969098915948 &   0.786 &  0.494298816339735\\
4 & 0.660960180256631 &   0.786 &  0.494116056132206\\
5 & 0.660952197968529 &   0.786 &  0.494109730708665\\
6 & 0.660952010967773 &   0.786 &  0.494109338690037\\
7 & 0.660952005398424 &   0.786 &  0.494109316007370\\
8 & 0.660952004871061 &   0.786 &  0.494109315475798\\
9 & 0.660952004869326 &   0.786 &  0.494109315436505\\
10 &   0.660952004868639 &   0.786 &  0.494109315435604\\
11 &   0.660952004868671 &   0.786 &  0.494109315435619\\
12 &   0.660952004868692 &   0.786 &  0.494109315435650\\
\hline
\end{tabular}
\medskip

\caption{Computed values of $\varrho_1(\alpha_0,0)$ and $\varrho_1(\alpha_0,0)$ with polynomial degree $Q$ on $10$ elements in $(-5,5)$. The numerical value of $\alpha_0$ is also provided.}
\label{T1}
\end{table}

According to the analysis of \cite{HP10}, the bottom of the spectrum $\inf \Spec(\mathcal{M}^{[2]})$ of the Montgomery operator of order two coincides with $\varrho_1(0,0)$. We can see in Figure \ref{F2}, that $\alpha=0$ is a local \emph{maximum} of the function $\alpha\mapsto\varrho_1(\alpha,0)$. Table \ref{T1} provides the value $0.660952004868639$ for $\varrho_1(0,0)$ (with presumably $13$ correct digits), and the value $0.786$ for $\alpha_0$ with 3 correct digits. Refining the sampling of $\alpha$ by $101$ values in the interval $[0.786,0.787]$ yields
\[
   \alpha_0 = 0.78628\quad\mbox{and}\quad \varrho_1(\alpha_0,0)=0.49410921120.
\]

\subsection{Asymptotic analysis in the small angle limit}
In this section, we prove Theorem \ref{1.8Intro}. The presentation is mostly inspired by \cite[Section 2]{RBH} and we just highlight the most important steps and differences.

\subsubsection{Changes of variables}

We are interested in the behavior of the first eigenpair of $\Cr$ (defined in \eqref{eq:Xeps}) as $\varepsilon\to0$. To investigate this, we perform two changes of variables. First, the scaling
\begin{equation}\label{chgmtechelle}
 (\sigma,\tau) \longrightarrow (s,t) \quad\mbox{so that}\quad
   s=\varepsilon \sigma, \quad t=\tau, 
\end{equation}
brings the spectral analysis of the operator $\Cr$ to the following unitarily equivalent operator
\begin{equation}
\label{eq:cLeps}
   \cL_\varepsilon=D_{t}^2+\left(\varepsilon D_{s}+s^2t-\frac{t^3}{3}\right)^2\,.
\end{equation}
Second, we localize in $s$ around a point $\alpha_0$ such that there exists a value $\xi_0$ for which $\varrho_1$ reaches its minimum $\rS_0$ in $(\alpha_0,\xi_0)$.

By the new change of variable
\begin{equation}\label{chgmtechelle2}
  (s,t) \longrightarrow (\fs,\ft) \quad\mbox{so that}\quad
   s=\alpha_0+\varepsilon^{1/2}\fs, \quad t=\ft, 
\end{equation}
and a gauge transform, the operator $\cL_\varepsilon$ becomes
\begin{equation}
\label{eq:fL}
\fL_{\varepsilon}=D_{\ft}^2+\left(\xi_0+\varepsilon^{1/2} D_{\fs}
+\varepsilon\fs^2\ft+2\varepsilon^{1/2}\alpha_0
\fs\ft+\alpha_0^2\ft-\frac{\ft^3}{3}\right)^2.\end{equation}
Thus the above three operators have the same eigenvalues
\[
   \Spec(\Cr) = \Spec(\cL_\varepsilon) = \Spec(\fL_{\varepsilon}).
\]

\begin{prop}
\label{prop:locS0}
For all $N\ge1$, there exist $C_N>0$ and $\varepsilon_N>0$ such that for all $\varepsilon\in (0,\varepsilon_N)$ there exist at least $N$ eigenvalues (counted with multiplicity) of the operator $\Cr$ contained in the ball of radius $C_N\,\varepsilon$ centered at $\rS_0$. 
\end{prop}

\begin{proof}
 We can write
\begin{equation*}
   \fL_{\varepsilon}=\fL_{0} +\varepsilon^{1/2}\fL_{1}+\mathcal{O}(\varepsilon),
\end{equation*}
with
\begin{equation*}
   \fL_{0}=\sX_{\alpha_0, \xi_0} \quad\mbox{and}\quad
   \fL_{1}
   = 2\left(\xi_{0}+\alpha_{0}^2\ft-\frac{\ft^3}{3}\right)\left(D_{\fs}+2\alpha_0\fs\ft\right).
\end{equation*}
We are looking for quasimodes in the form
$$\psi=\psi_{0}+\varepsilon^{\frac{1}{2}}\psi_{1} \quad \text{and} \quad \varkappa=\varkappa_{0}+\varepsilon^{\frac{1}{2}}\varkappa_{1},$$
such that
$\fL_\varepsilon\psi =\varkappa\psi+\mathcal{O}(\varepsilon)$ is satisfied.
Gathering the terms in $\varepsilon^{0}$ we get the equation 
\[
   \sX_{\alpha_0,\xi_0}\psi_0=\varkappa_0\psi_0\,.
\]
Thus $\varkappa_{0}$ is in the spectrum of $\sX_{\alpha_{0}, \xi_{0}}$ and that $\psi_0$ is an associated eigenfunction. We choose
\begin{equation*}
   \varkappa_{0}=\rS_0,
\end{equation*} 
and we take $\psi_{0}$ (unitary) in the following form 
\begin{equation}\label{DSFchoix2croix}
   \psi_0(\fs,\ft)=f_0(\fs)\,u_0(\ft)\,,
\end{equation} 
with $u_0=u_{\alpha_0,\xi_0}$ and $f_0$ in the Schwartz class. 

Gathering the terms in $\varepsilon^{1/2}$, we get the equation 
\begin{equation}
\label{eq:fL1}
   (\fL_{0}-\rS_0)\psi_{1}
   =-(\fL_{1} - \varkappa_{1})\psi_{0}\,.
\end{equation}
We have
\[
   \fL_{1}\psi_{0}(\fs,\ft)=
   \left(\xi_{0}+\alpha_{0}^2\ft-\frac{\ft^3}{3}\right)
   \Big(2 u_{\alpha_{0},\xi_{0}}(\ft) \, (D_{\fs}f_{0})(\fs)
   +4\alpha_0\ft\, u_{\alpha_{0},\xi_{0}}(\ft)\, \fs f_{0}(\fs)\Big)\,.\]
At this point, we introduce the notation
\[
   \left(\partial_{\alpha}u\right)_{\alpha_0,\xi_0}=
   \left(\partial_{\alpha}u_{\alpha,\xi}\right)_{(\alpha,\xi)=(\alpha_0,\xi_0)} 
   \quad \text{and} \quad 
   \left(\partial_{\xi}u\right)_{\alpha_0,\xi_0}=
   \left(\partial_{\xi}u_{\alpha,\xi}\right)_{(\alpha,\xi)=(\alpha_0,\xi_0)}\,.
\]
Taking the derivative of  the equation $\sX_{\alpha,\xi}u_{\alpha,\xi} = \varrho_1(\alpha,\xi)\,u_{\alpha,\xi}$ with respect to $\alpha$ and $\xi$, we obtain
\begin{subequations}
\begin{align}
   \big(\sX_{\alpha,\xi} - \varrho_1(\alpha,\xi)\big) \left(\partial_{\alpha}u\right)_{\alpha,\xi} &=
   \left(\partial_\alpha \varrho_1(\alpha,\xi) - 4\alpha\ft
   \left(\xi_{0}+\alpha_{0}^2\ft-\frac{\ft^3}{3}\right) \right) u_{\alpha,\xi} \\
   \big(\sX_{\alpha,\xi} - \varrho_1(\alpha,\xi)\big) \left(\partial_{\xi}u\right)_{\alpha,\xi} &=
   \left(\partial_\alpha \varrho_1(\alpha,\xi) - 2
   \left(\xi_{0}+\alpha_{0}^2\ft-\frac{\ft^3}{3}\right) \right) u_{\alpha,\xi} \,.
\end{align}
\end{subequations}
In particular, for $(\alpha,\xi)$ equal to the critical point $(\alpha_0,\xi_0)$ we find that
\[
\begin{aligned}
   (\fL_{0}-\rS_0) \left(\partial_{\alpha}u\right)_{\alpha_0,\xi_0} &=
   - 4\alpha\ft
   \big(\xi_{0}+\alpha_{0}^2\ft-\tfrac{\ft^3}{3}\big) u_{\alpha_0,\xi_0} \\
   (\fL_{0}-\rS_0) \left(\partial_{\xi}u\right)_{\alpha_0,\xi_0} &=
   - 2
   \big(\xi_{0}+\alpha_{0}^2\ft-\tfrac{\ft^3}{3}\big) u_{\alpha_0,\xi_0} .
\end{aligned}
\]
Thus we have obtained explicit solutions of equation \eqref{eq:fL1} as:
\begin{equation}\label{DSFimpose2croix}
   \psi_{1}(\fs,\ft) = (D_{\fs}f_{0})(\fs) \, \left(\partial_{\alpha}u\right)_{\alpha_0,\xi_0}\!(\ft)
   + \fs f_{0}(\fs) \, \left(\partial_{\xi}u\right)_{\alpha_0,\xi_0}\!(\ft)
   \quad\mbox{with}\quad
   \varkappa_{1}=0\,.
\end{equation}
Therefore, for any function $\psi$ in the form
$$\psi=\psi_0+\varepsilon^{1/2}\psi_1,$$
where $\psi_0$ and $\psi_1$ are respectively given by in \eqref{DSFchoix2croix} and \eqref{DSFimpose2croix}, we have 
\[\| \left(\fL_\varepsilon-\rS_0\right)\psi\| \le C\varepsilon\| \psi\|\,.\]
In the construction procedure of $\psi_0$ and $\psi_1$, the function $f_{0}$ is left undetermined. Therefore the above estimate holds on finite dimensional spaces of arbitrary dimensions, which ends the proof of the proposition.
\end{proof}

\subsubsection{Localization estimates}
The following lemma is crucial to follow the strategy in \cite[Section 2]{RBH}. Note in particular that \cite[Assumption 1.7]{RBH} is not obviously checked in the present context.
\begin{lem}
\label{lem:gapunif}
There exists $C_{2}>0$ such that, for all $\varepsilon\in(0,1)$, and all $\psi$ in the form domain of $\Cr$,
\begin{equation}\label{eq.cinfty}
   \|m_{\varepsilon}(\sigma,\tau)^{1/4}\psi\|^2 \le C_2\left(\langle\Cr\psi,\psi\rangle + \|\psi\|^2\right)\,,
\end{equation}
where
\[m_{\varepsilon}(\sigma,\tau)=|\varepsilon^2\sigma^2-\tau^2|+2|\sigma|\varepsilon^2+2|\tau|\,.\]
\end{lem}

\begin{proof}
The lemma follows from \cite[Th\'eor\`eme (1.1)]{HM88}. We only have to check the assumptions of the theorem and the uniformity with respect to $\varepsilon$.
Here the magnetic field has one component $B=\varepsilon^2\sigma^2-\tau^2$. We take the integer $r$ used there as $1$. The condition \cite[(1.9)]{HM88} is trivially satisfied for $C_1=8$
\[
   2\varepsilon^2+2 \le C_1(m(\sigma,\tau)+1)\,.
\]
Then \cite[(1.11)]{HM88} tells us that
\[
   \forall\psi\in\mathcal{C}^\infty_0(\R^2),\quad
   \|m(\sigma,\tau)^{1/4}\psi\|^2 \le C_2\left(\langle\Cr\psi,\psi\rangle + \|\psi\|^2\right)\,.
\]
A careful check of the proof of \cite[Th\'eor\`eme (1.1)]{HM88}, involving only commutators of analytic vector fields with respect to $\varepsilon$, shows that $C_2$ does not depend on $\varepsilon\in(0,1)$. The extension of this estimate to the form domain follows by density.
\end{proof}

We can reformulate Lemma \ref{lem:gapunif} in terms of the operator $ \cL_\varepsilon$ and deduce the following.
\begin{prop}
\label{prop:gapunif2}
There exists $C_{2}>0$ such that, for all $\varepsilon\in(0,1)$, and all $\psi$ in the form domain of $\cL_{\varepsilon}$,
\begin{equation}
   \|\left[|s^2-t^2|+2|t|\right]^{1/4}\psi\|^2 \le C_2\left(\langle \cL_\varepsilon\psi,\psi\rangle + \|\psi\|^2\right)\,.
\end{equation}
In particular, for all $\rS_0^*>\rS_0$ there exists $R_\ast >0$ such that, for all $\varepsilon\in(0,1)$,
\[\cL_\varepsilon^{\Dir,R_\ast}\ge \rS_0^*\,,\]
where $\cL_\varepsilon^{\Dir,R_\ast}$ is the Dirichlet realization of the operator $\cL_\varepsilon$ outside the ball of center $0$ and radius $R_\ast$.
\end{prop}
\begin{proof}
To deduce the second assertion, it suffices to note that
$\displaystyle{\lim_{|s|+|t|\to+\infty}|s^2-t^2|+2|t|=+\infty}$.
\end{proof}
It is classical to deduce from Proposition \ref{prop:gapunif2} that the eigenfunctions associated with the low lying eigenvalues satisfy Agmon estimates with respect to $(s,t)$ (see, for instance, \cite[Section 2.2]{RBH}). By taking derivatives of the eigenvalue equation, we finally get the following corollary.

\begin{cor}\label{petitemicro}
Let $C_0$ and $k,l,d\in \mathbb{N}$. There exist $\varepsilon_0>0, C>0$, and $c_0>0$ such that for all $\varepsilon\in(0,\varepsilon_0)$ and all eigenpairs $(\varkappa,\psi)$ of the operator $\cL_\varepsilon$ with $\varkappa\le \rS_0+C_0\varepsilon$, we have
\[\| t^k s^l\psi\| \le C\| \psi\|, \quad \mathcal{Q}_\varepsilon(t^k s^l\psi) \le C\| \psi\|^2,\]
\[\| (D_t)^d s^l t^k\psi\| \le C\| \psi\|, \quad \| (\varepsilon D_s)^d s^l t^k\psi)\| \le C\| \psi\|\,.\]
\end{cor}

\subsubsection{Coherent states}
Following a classical formalism, see \cite{B71, CR12} for instance, we introduce the annihilation operator $a$ 
and the creation operator $a^*$
\[a=\frac{1}{\sqrt{2}}\left(\fs+\partial_{\fs}\right)\,,\quad a^*=\frac{1}{\sqrt{2}}\left(\fs-\partial_{\fs}\right)\,.\]
We have the following identities
\begin{equation}\label{eq.sds}
\fs=\frac{a+a^*}{\sqrt{2}}\,, \quad \partial_{\fs}=\frac{a-a^*}{\sqrt{2}}\,,\quad
\mbox{and}\quad[a,a^*]=1.
\end{equation}
Setting $g_{0}(\fs) = \pi^{-1/4} e^{-\fs^2/2}$,
we introduce the coherent states  for any $u\in\R$ and $p\in\R$, 
\[f_{u,p}(\fs)=e^{ip\fs}g_{0}(\fs-u)\,,\]
and for all $\psi \in L^2(\R)$, the associated projection defined by
\begin{equation*}
\Pi_{u,p}\psi=\big\langle \psi,f_{u,p}\big\rangle_{L^2(\R,\rd\fs)} f_{u,p}\,.
\end{equation*}
We have the resolution of the identity
\[\psi=\int_{\R^2} \Pi_{u,p}\psi\,\rd u\rd p\,,\]
and the Parseval formula
\[\|\psi\|^2=\int_{\R}\int_{\R^2} |\Pi_{u,p}\psi|^2\,\rd u\rd p\,\rd \fs\,.\]

\begin{lem}[\cite{B71}]\label{lem3}For all $n,m \in \mathbb{N}$, we have
\begin{equation*}
   (a)^m(a^*)^n=
   \int_{\R^2} 
   \left(\frac{u+ip}{\sqrt{2}}\right)^m\left(\frac{u-ip}{\sqrt{2}}\right)^n\Pi_{u,p}\,\rd u\rd p.
\end{equation*}
\end{lem}

Now we have all tools at hand to end the proof of Theorem \ref{1.8Intro}. Taking Proposition \ref{prop:locS0} into account, it remains to prove the following lemma:
\begin{lem}
\label{Lfin}
Consider an eigenfunction $\Psi_\varepsilon$ of $\Cr$ associated with an eigenvalue $\varkappa(\varepsilon)\leq \rS_{0}+C_0\varepsilon$. Then there exists $C$ only depending on $C_0$ such that
\begin{equation}
\label{eq:fin}
   \varkappa(\varepsilon)\|\Psi_\varepsilon\|^2 = \langle \Cr\Psi_\varepsilon,\Psi_\varepsilon\rangle
   \geq (\rS_{0}-C\varepsilon)\|\Psi_\varepsilon\|^2.
\end{equation}
\end{lem}

\begin{proof}
In view of \eqref{eq:cLeps}, it suffices to prove \eqref{eq:fin} for an eigenpair $(\varkappa(\varepsilon),\psi_\varepsilon)$ of the operator $\cL_{\varepsilon}$.
By the scaling $s=\varepsilon^{\frac{1}{2}}\fs$, $\cL_{\varepsilon}$ becomes the unitarily equivalent operator
\[\widetilde{\fL}_{\varepsilon}=D_{\ft}^2+\left(\varepsilon^{1/2} D_{\fs}
+\varepsilon\fs^2\ft-\frac{\ft^3}{3}\right)^2,\]
which we expand as
\[
  \widetilde{\fL}_{\varepsilon}=\widetilde{\fL}_{0}
  +\varepsilon^{1/2}\widetilde{\fL}_{1}
  +\varepsilon \widetilde{\fL}_{2}+\varepsilon^{3/2}\widetilde{\fL}_3
  +\varepsilon^{2}\widetilde{\fL}_4,
\]
where
\begin{equation*}
  \widetilde{\fL}_{0}=D^2_{\ft}+\frac{\ft^6}{9},\:\:\:
  \widetilde{\fL}_{1}=-\frac{2\ft^3}{3}D_{\fs},\:\:\: 
  \widetilde{\fL}_{2}=-\frac{2}{3}\ft^4\fs^2+D^2_{\fs},\:\:\:
  \widetilde{\fL}_{3}=\ft\left(\fs^2D_{\fs}+D_{\fs}\fs^2\right),\:\:\:
  \widetilde{\fL}_{4}=\fs^4\ft^2.
\end{equation*}
We use \eqref{eq.sds} and we commute $a$ and $a^*$ to put all the $a^*$ on the right. With Lemma \ref{lem3}, we get that there exist $(\gamma,\delta)\in\R^2$,  a homogeneous polynomial of order $1$, $L(X,Y)$, and a non-commutative homogeneous polynomial of order $2$, $P(X,Y)$, such that
\begin{equation}\label{eq:Lepst}
  \widetilde{\fL}_{\varepsilon}=
  \fL_{\varepsilon}^{\mathsf{W}}
  +\varepsilon\fL_{2,{\sf{rem}}}
  +\varepsilon^{3/2}\fL_{3,\sf{rem}}
  +\varepsilon^2\fL_{4,\sf{rem}}\,,
\end{equation} 
where
\begin{equation*}
  \fL_{\varepsilon}^{\mathsf{W}}=
  \int_{\R^2} \sX_{u\sqrt\varepsilon, \:p\sqrt\varepsilon}\,\Pi_{u,p} \,\rd u\rd p\,,
\end{equation*}
and
\begin{equation*}
  \fL_{2,\sf{rem}}=\gamma \ft^4+\delta,\quad 
  \fL_{3,\sf{rem}}=\ft L(a,a^*), \quad 
  \fL_{4,\sf{rem}}=\ft^2 P(a,a^*)\,.
\end{equation*}
Consider an eigenfunction $\widetilde{\psi}_{\varepsilon}$ of $\widetilde{\fL}_{\varepsilon}$ associated with an eigenvalue $\varkappa(\varepsilon)\leq \rS_{0}+C_0\varepsilon$. From \eqref{eq:Lepst} and using the quadratic forms $\widetilde{\mathfrak{Q}}_\varepsilon$ of $\widetilde{\fL}_{\varepsilon}$ and $\sQ_{\alpha,\xi}$ of $\sX_{\alpha,\xi}$, we get
\begin{multline}\label{4.28}
   \widetilde{\mathfrak{Q}}_\varepsilon(\widetilde{\psi}_{\varepsilon})\geq 
   \int_{\R}\int_{\R^2} \sQ_{u\sqrt\varepsilon, \:p\sqrt\varepsilon}
   (\Pi_{u,p}\widetilde{\psi}_{\varepsilon})\,\rd u\rd p\,\rd\fs \\
   - \varepsilon\|(\gamma\ft^4+\delta)\widetilde{\psi}_{\varepsilon}\|\,\|\widetilde{\psi}_{\varepsilon}\|
   -\varepsilon^{3/2}\| L(a,a^*)\widetilde{\psi}_{\varepsilon}\|\, \|\ft\widetilde{\psi}_{\varepsilon}\|
   -\varepsilon^2\|P( a,a^*)\widetilde{\psi}_{\varepsilon}\|\,\|\ft^2\widetilde{\psi}_{\varepsilon}\|.
\end{multline}
Since $\rS_0$ is a minimum for the band function of $\sX_{\alpha,\xi}$, we have, for each $(u,p)\in\R^2$ and $\fs\in\R$, the bound from below
\[
   \sQ_{u\sqrt\varepsilon, \:p\sqrt\varepsilon}
   (\Pi_{u,p}\widetilde{\psi}_{\varepsilon})(\fs,\cdot) \ge 
   \rS_0\|\Pi_{u,p}\widetilde{\psi}_{\varepsilon}(\fs,\cdot)\|^2_{L^2(\R,\rd\ft)}\,.
\]
Thus by Parseval formula
\begin{equation}
\label{4.29}
   \int_{\R}\int_{\R^2} \sQ_{u\sqrt\varepsilon, \:p\sqrt\varepsilon}
   (\Pi_{u,p}\widetilde{\psi}_{\varepsilon})\,\rd u\rd p\,\rd\fs \ge
   \rS_0\|\widetilde{\psi}_{\varepsilon}\|^2_{L^2(\R^2,\rd\fs\rd\ft)}.
\end{equation}
Let us estimate the remainder terms in \eqref{4.28}. The homogeneity of the polynomials $P$ and $L$ yields
\[
   \varepsilon^{3/2} L(a,a^*)\widetilde{\psi}_{\varepsilon} =
   \varepsilon L(\varepsilon^{1/2}a,\varepsilon^{1/2}a^*)\widetilde{\psi}_{\varepsilon}
   \quad\mbox{and}\quad
   \varepsilon^{2} P(a,a^*)\widetilde{\psi}_{\varepsilon} =
   \varepsilon P(\varepsilon^{1/2}a,\varepsilon^{1/2}a^*)\widetilde{\psi}_{\varepsilon}.
\]
Thus, using the scaling $s=\varepsilon^{1/2}\fs$ and Corollary \ref{petitemicro}, we find
\[
   \varepsilon\|(\gamma\ft^4+\delta)\widetilde{\psi}_{\varepsilon}\|\,\|\widetilde{\psi}_{\varepsilon}\|
   +\varepsilon^{3/2}\| L(a,a^*)\widetilde{\psi}_{\varepsilon}\|\, \|\ft\widetilde{\psi}_{\varepsilon}\|
   +\varepsilon^2\|P( a,a^*)\widetilde{\psi}_{\varepsilon}\|\,\|\ft^2\widetilde{\psi}_{\varepsilon}\|
   \le C \varepsilon\|\widetilde{\psi}_{\varepsilon}\|^2.
\]
From this, combined with \eqref{4.28} and \eqref{4.29}, follows that
\[\varkappa(\varepsilon)\|\widetilde{\psi}_{\varepsilon}\|^2=\widetilde{\mathfrak{Q}}_\varepsilon(\widetilde{\psi}_{\varepsilon})\geq (\rS_{0}-C\varepsilon)\|\widetilde{\psi}_{\varepsilon}\|^2.\]
Coming back to the operator $\cL_{\varepsilon}$ and finally to the operator $\Cr$, this ends the proof of Lemma \ref{Lfin}, and thus of Theorem \ref{1.8Intro}.
\end{proof}

\subsubsection{Computation of the ground states of $\Cr$}
In the last section of this paper, we present computations of the first eigenpair of the operator $\Cr$ for a decreasing sequence of values of $\varepsilon$. Let us agree that
\[
   \varepsilon_\ell = 2^{-1-\ell/2},\quad \ell\ge0
\]
so that $\varepsilon_0=1/2$, $\varepsilon_1=\sqrt{2}/4$, $\varepsilon_2=1/4$, etc\ldots
We have computed the first eigenpair of $\Cr$ for $\varepsilon=\varepsilon_\ell$, $\ell\in\{0,1,\ldots,12\}$ with natural boundary conditions on the domain $[-a_\ell,a_\ell]\times[-8,8]$ (with $a_\ell = 4/\varepsilon_\ell = 8\cdot2^{\ell/2}$) by a finite element discretization on a uniform rectangular grid of $48\times6$ elements of partial degree $10$ in each variable.

Theorem \ref{1.8Intro} yields the convergence of $\varkappa_1(\varepsilon)$ to $\rS_0$ at a rate of $\mathcal{O}(\varepsilon)$. Our computations (see Figures \ref{F1} and \ref{F2}) suggest that $\rS_0\simeq0.4941$. In Figure \ref{F3} we plot the difference $\varkappa_1(\varepsilon_\ell)-0.4941$ versus $\varepsilon_\ell$ (here we use a $\log-\log$ scale in base 2).

\begin{figure}[ht]
\begin{center}
    \includegraphics[width=0.7\textwidth]{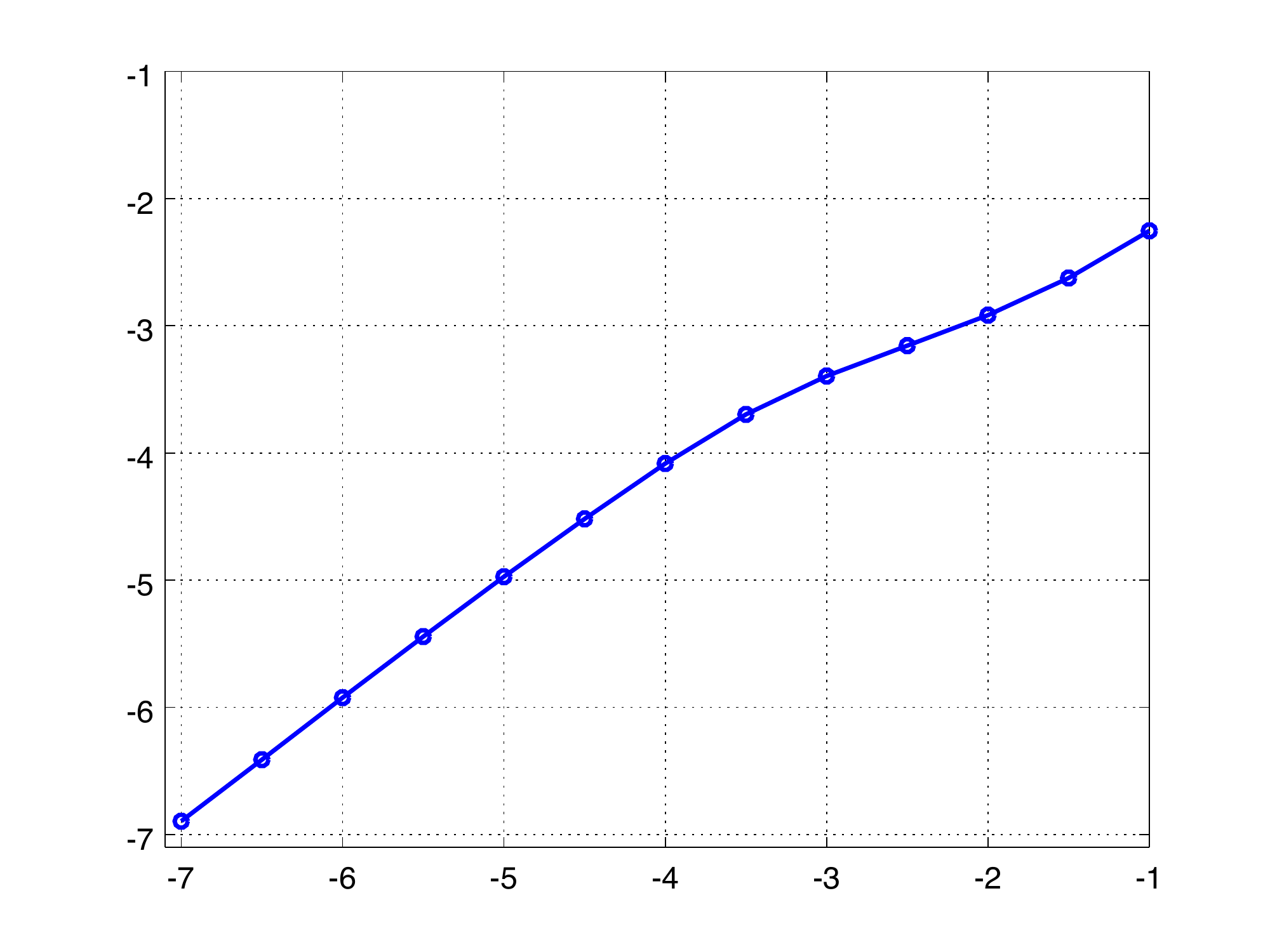}
    \end{center}
    \caption{Plot of $\log_2(\varkappa_1(\varepsilon_\ell)-0.495)$ versus $\log_2(\varepsilon_\ell)$.
    }
    \label{F3}
\end{figure}

Figure \ref{F4} gives the modulus of the first eigenvector of the operator $\Cr$ and a numerical value of the eigenvalue $\varkappa_1(\varepsilon)$, for $\varepsilon=\varepsilon_\ell$ with $\ell=0,\ldots,5$. 
For $\varepsilon=\varepsilon_\ell$, the horizontal scale (variable $\sigma$) is $[-a'_\ell,a'_\ell]$ with $a'_\ell = 5\cdot2^{\ell/2}$ and the vertical scale (variable $\tau$) is always $[-5,5]$. The proportion between the two scales is kept, which makes the plot of the vertical scale shrink as $\ell$ increases.

\begin{figure}[ht]
\captionsetup[subfloat]{labelformat=empty}
\begin{center}
\begin{tabular}{@{\hskip-5ex}c@{\hskip-4ex}c@{\hskip-4ex}c}
\subfloat[$\varkappa_1=0.7039$ \ ($\ell=0$)]{\includegraphics[width=6.4cm]{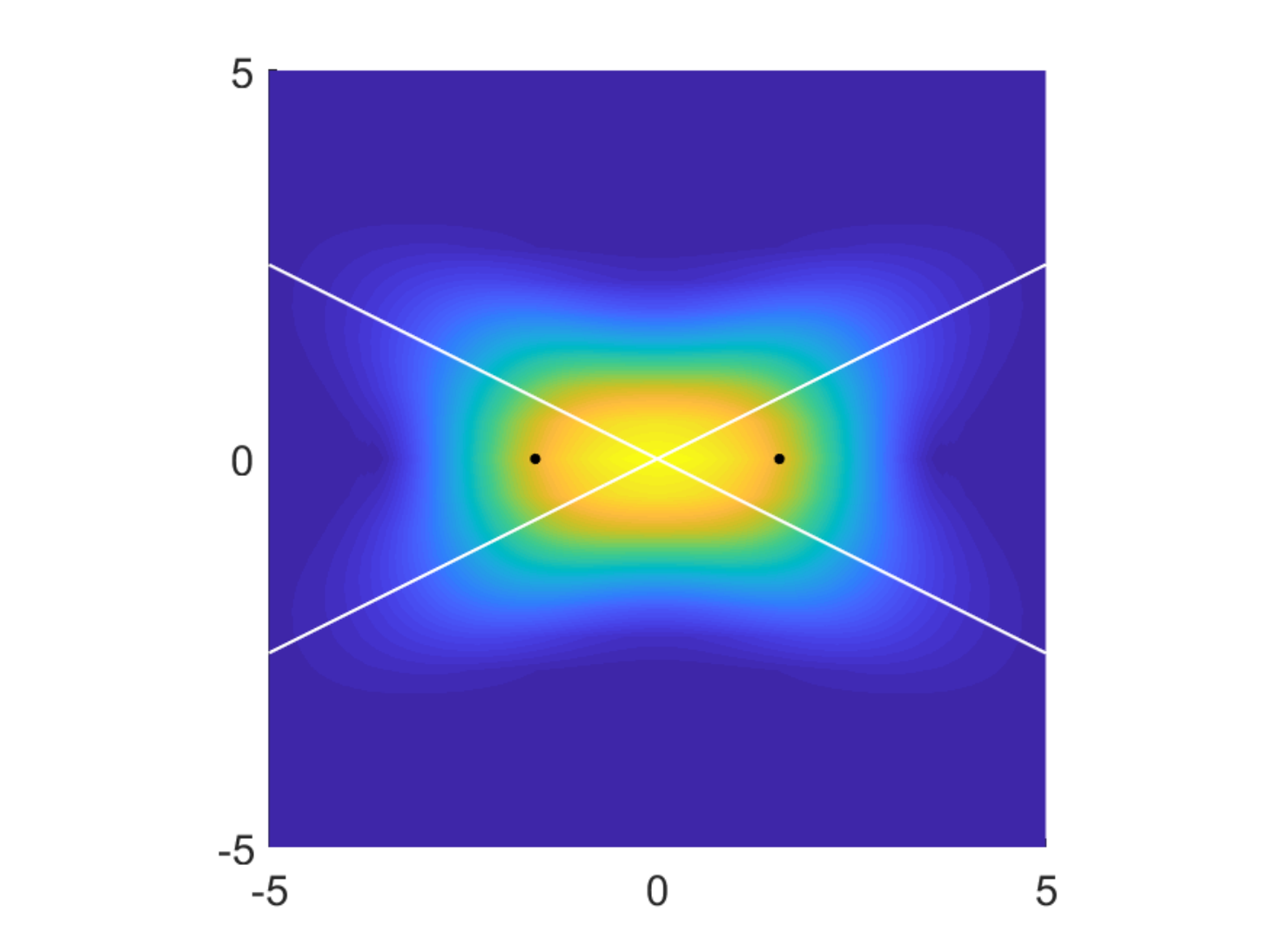}} & 
\subfloat[$\varkappa_1=0.6563$ \ ($\ell=1$)]{\includegraphics[width=6.4cm]{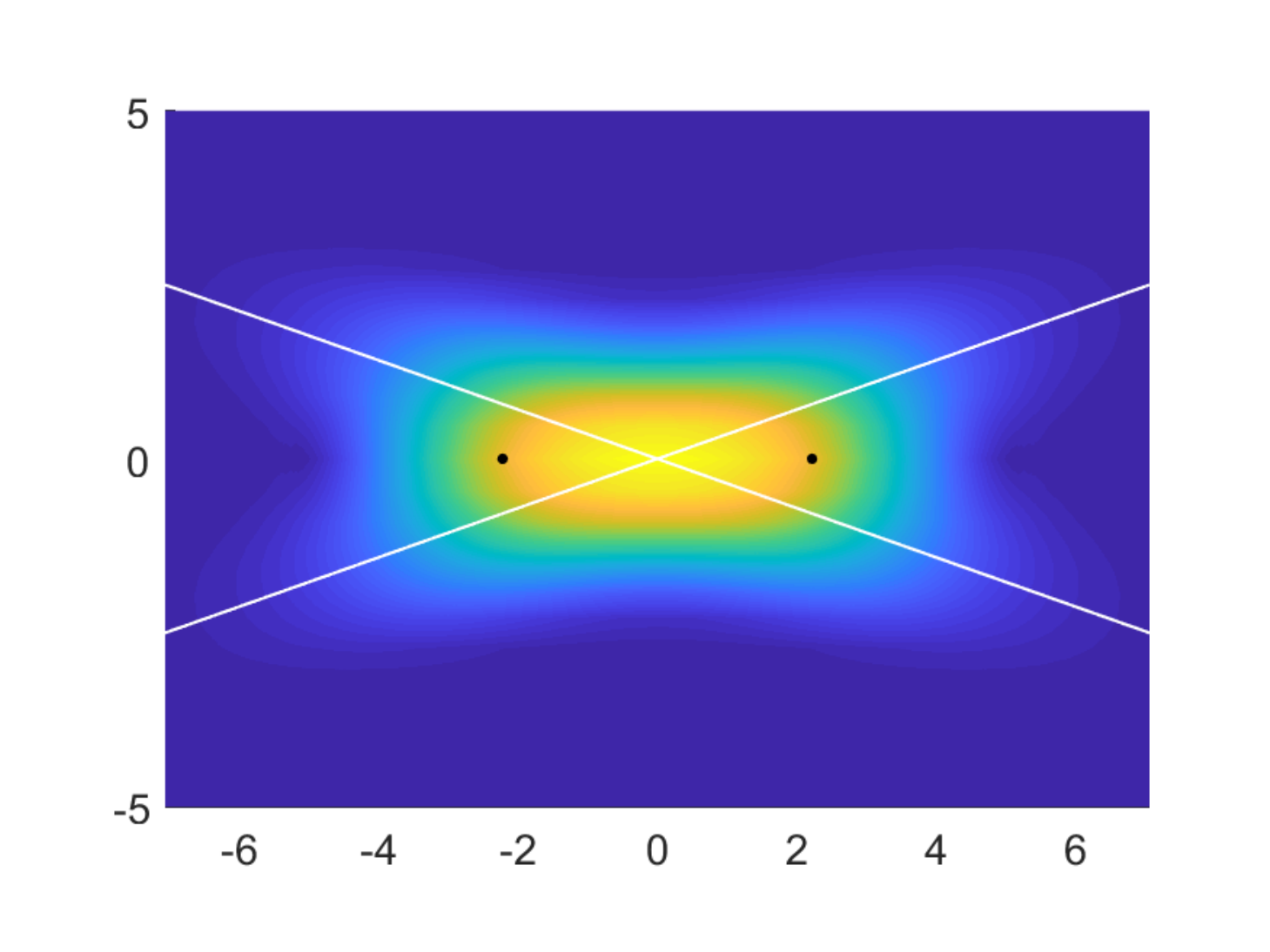}} & 
\subfloat[$\varkappa_1=0.6266$ \ ($\ell=2$)]{\includegraphics[width=6.4cm]{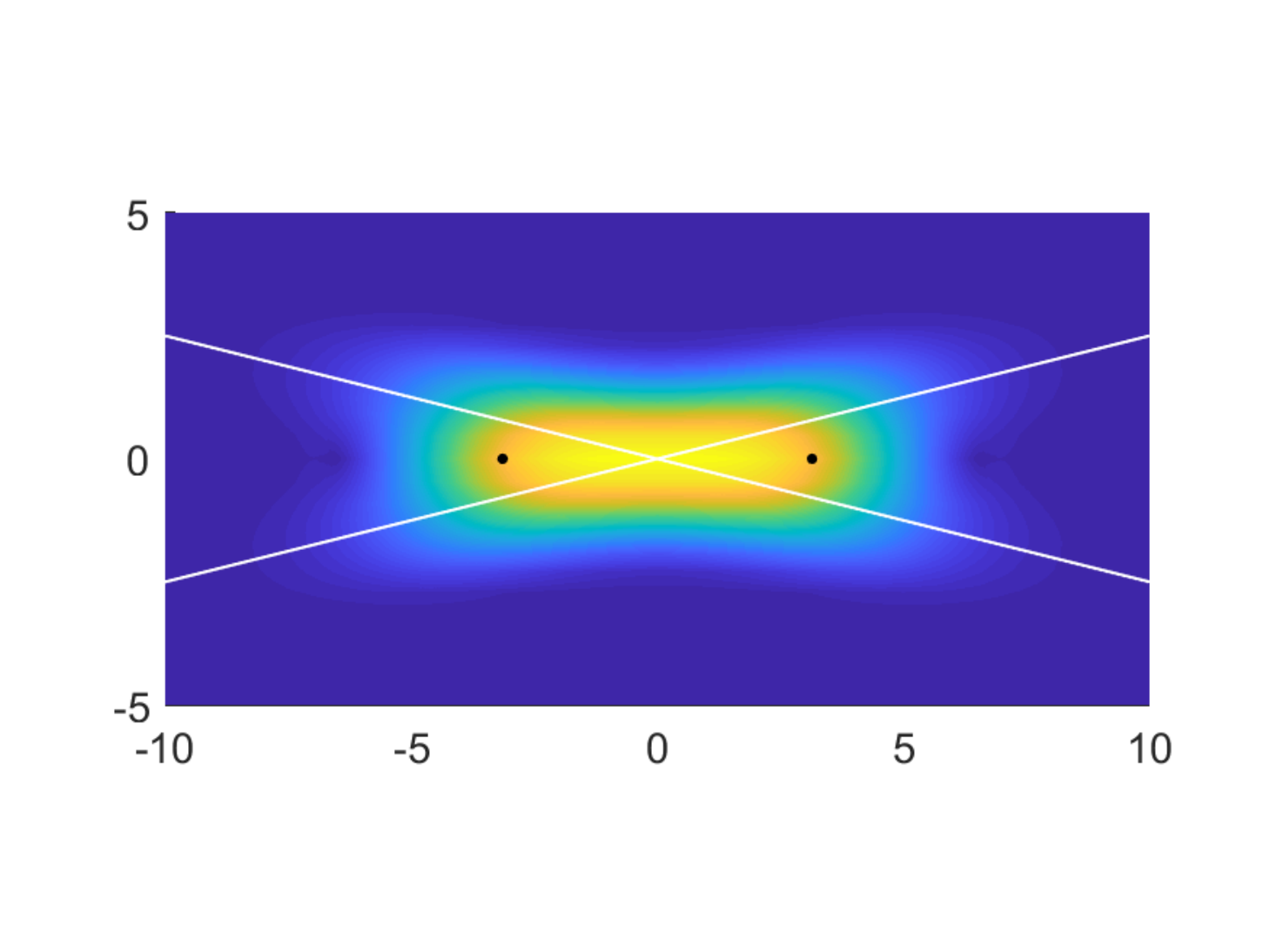}} \\ 
\subfloat[$\varkappa_1=0.6063$ \ ($\ell=3$)]{\includegraphics[width=6.4cm]{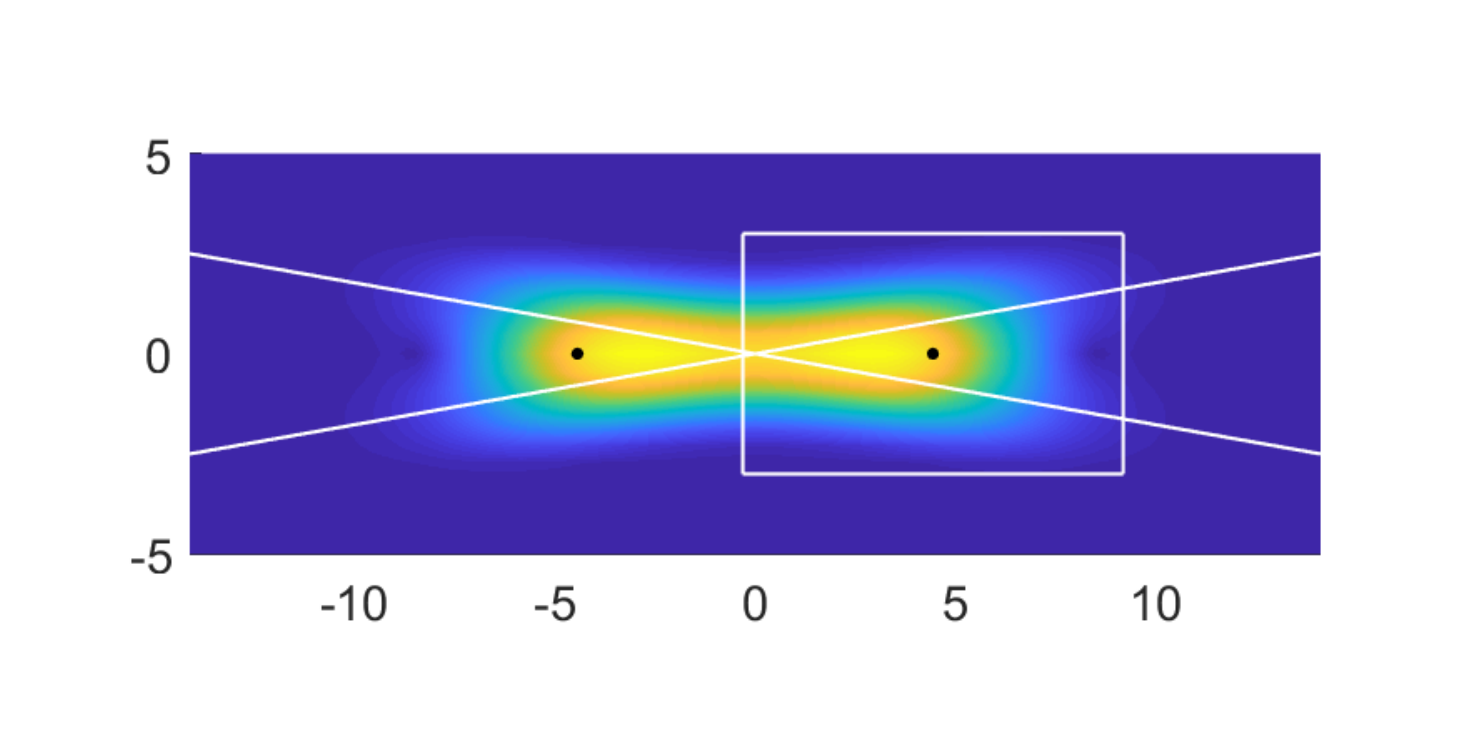}} & 
\subfloat[$\varkappa_1=0.5892$ \ ($\ell=4$)]{\includegraphics[width=6.4cm]{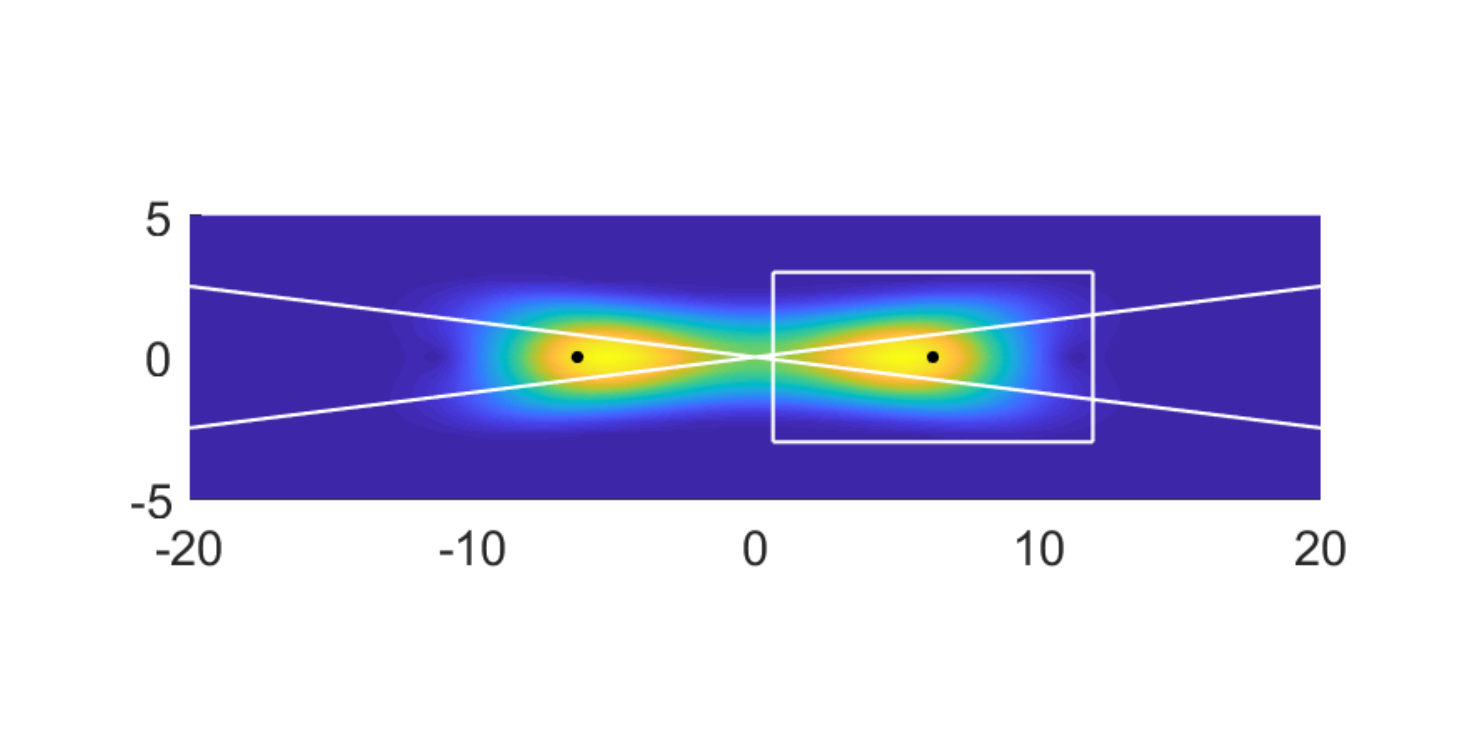}} &
\subfloat[$\varkappa_1=0.5712$ \ ($\ell=5$)]{\includegraphics[width=6.4cm]{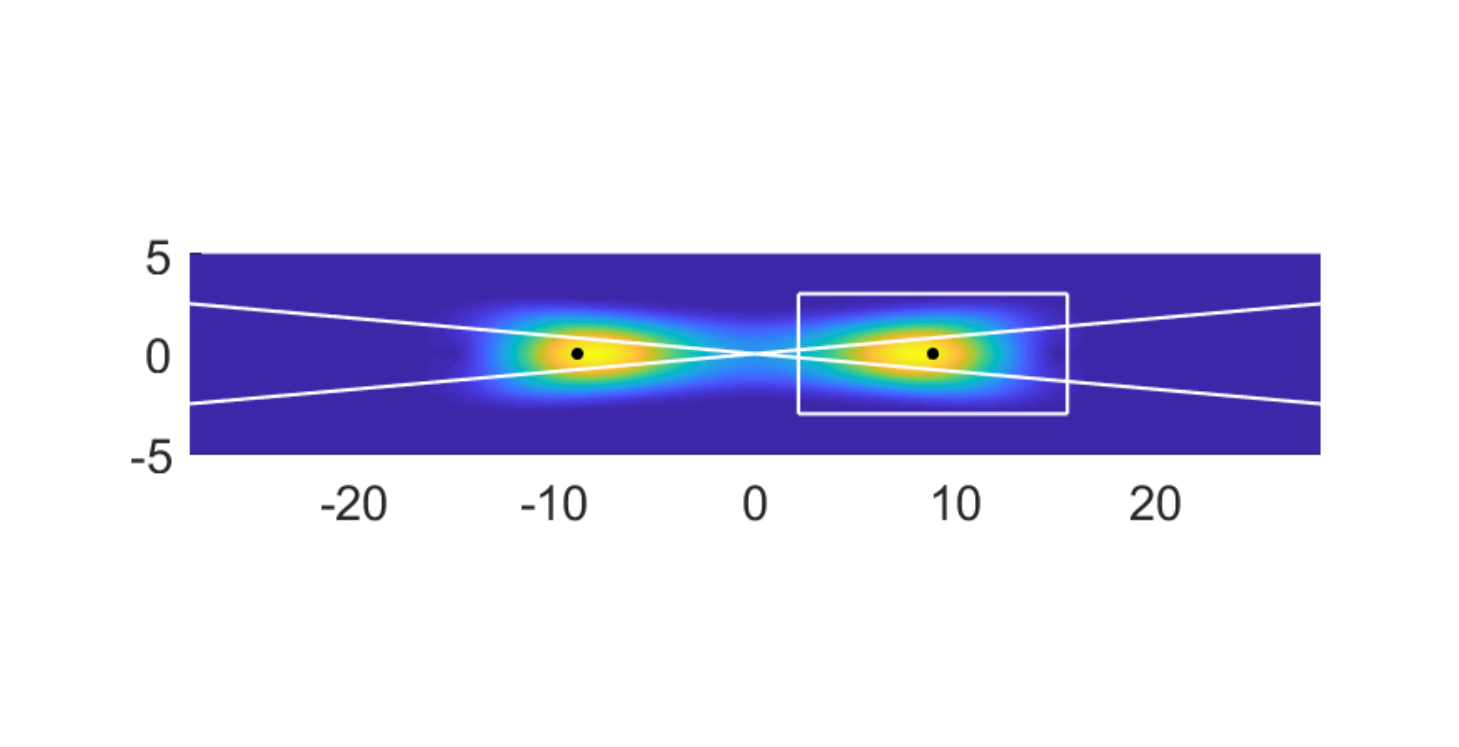}} \\ 
\end{tabular}
\end{center}
\caption{Modulus of the first eigenvector of $\mathcal{X}_{\varepsilon_\ell}$ for $\ell=0,\ldots,5$. The white cross represents the two lines of cancellation of the magnetic field and the black dots mark the two points $(\pm\alpha_0/\varepsilon_\ell,0)$.
The rectangular boxes in the second row indicate the zoom region displayed in Figure \ref{F5}.}
\label{F4}
\end{figure}

In Figure \ref{F5}, we zoom some region around the point $\alpha_0/\varepsilon$ of the plot of the first eigenvector of $\Cr$ when $\varepsilon=\varepsilon_\ell$ with $\ell=3,\ldots,10$. When $\ell=6,\ldots,10$, we have computed the eigenvector on the smaller region $[-\frac{a_\ell}{2},\frac{a_\ell}{2}]\times[-4,4]$ on a uniform rectangular grid of $48\times6$ elements of partial degree $10$ in each variable. We represent the modulus of the eigenvector on the region
\[
   \Big[\frac{\alpha_0}{\varepsilon_\ell} - \frac{2}{\sqrt{\varepsilon_\ell} } \,,\,
    \frac{\alpha_0}{\varepsilon_\ell} + \frac{2}{\sqrt{\varepsilon_\ell} }\Big]\times [-3,3]\,.
\]

\begin{figure}[ht]
\captionsetup[subfloat]{labelformat=empty}
\begin{center}
\begin{tabular}{@{\hskip-5ex}c@{\hskip-4ex}c@{\hskip-4ex}c@{\hskip-4ex}c}
\subfloat[$\varkappa_1=0.6063$ \ ($\ell=3$)]{\includegraphics[width=4.5cm]{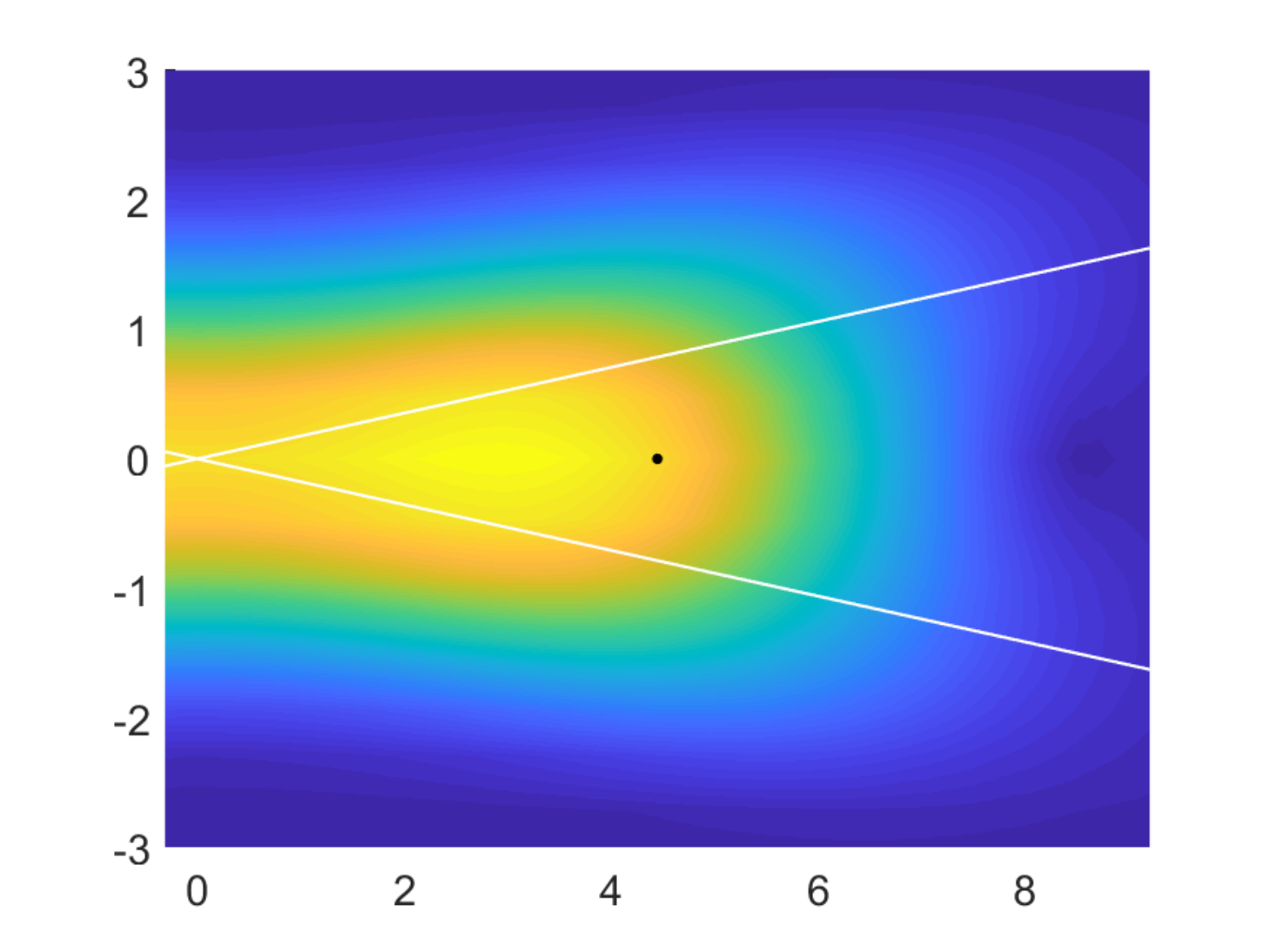}} & 
\subfloat[$\varkappa_1=0.5892$ \ ($\ell=4$)]{\includegraphics[width=4.5cm]{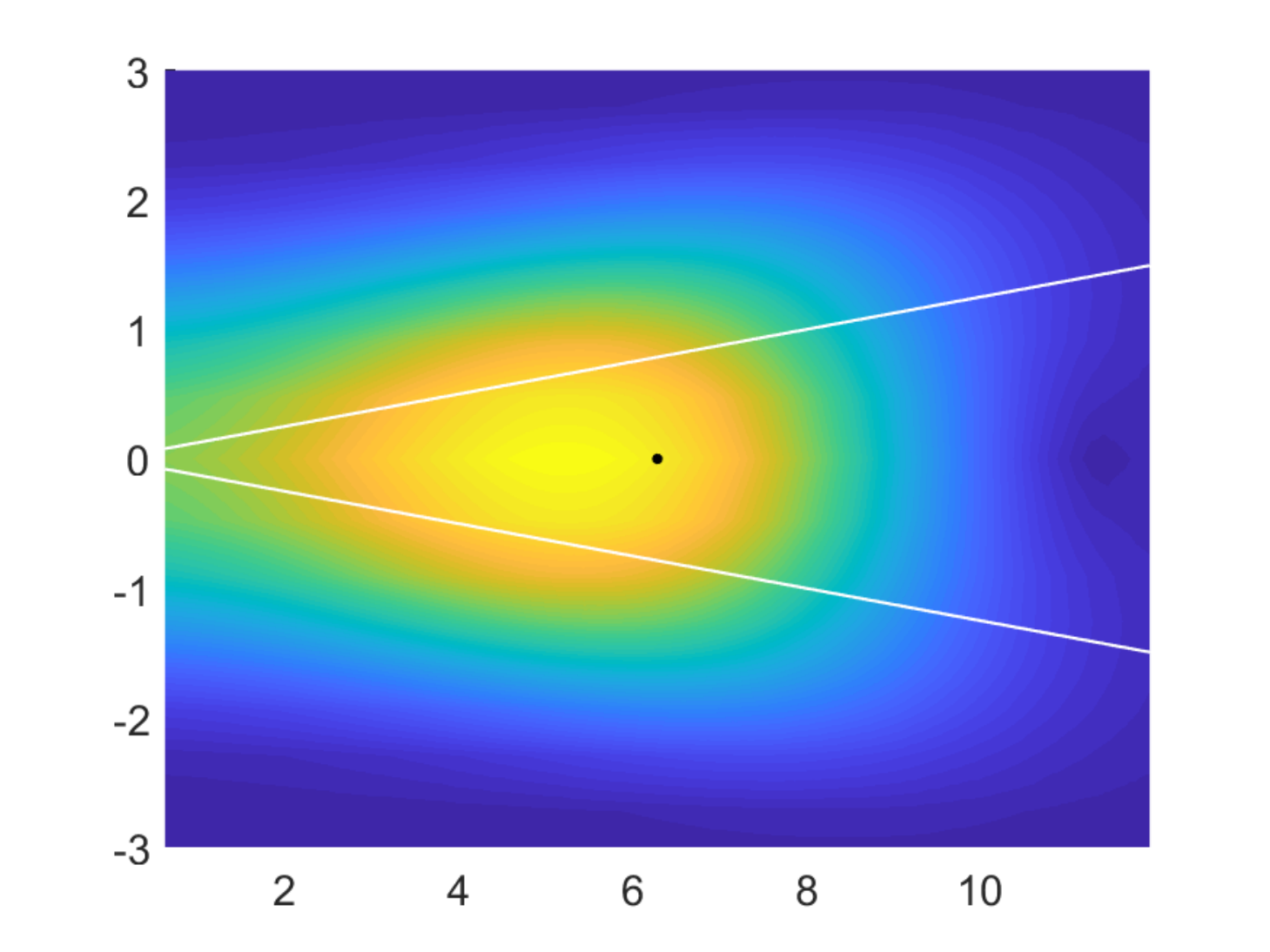}} &
\subfloat[$\varkappa_1=0.5712$ \ ($\ell=5$)]{\includegraphics[width=4.5cm]{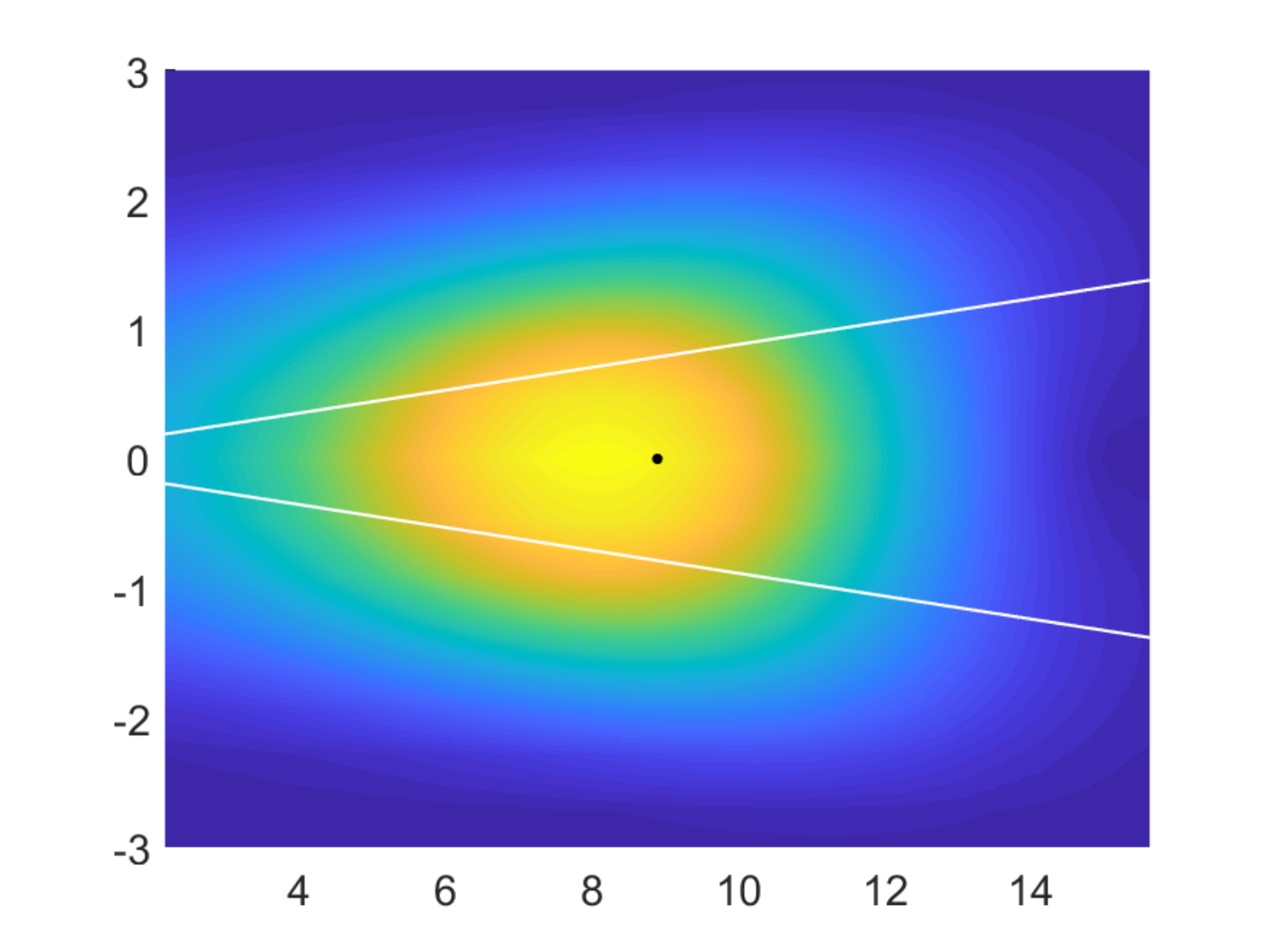}} &
\subfloat[$\varkappa_1=0.5531$ \ ($\ell=6$)]{\includegraphics[width=4.5cm]{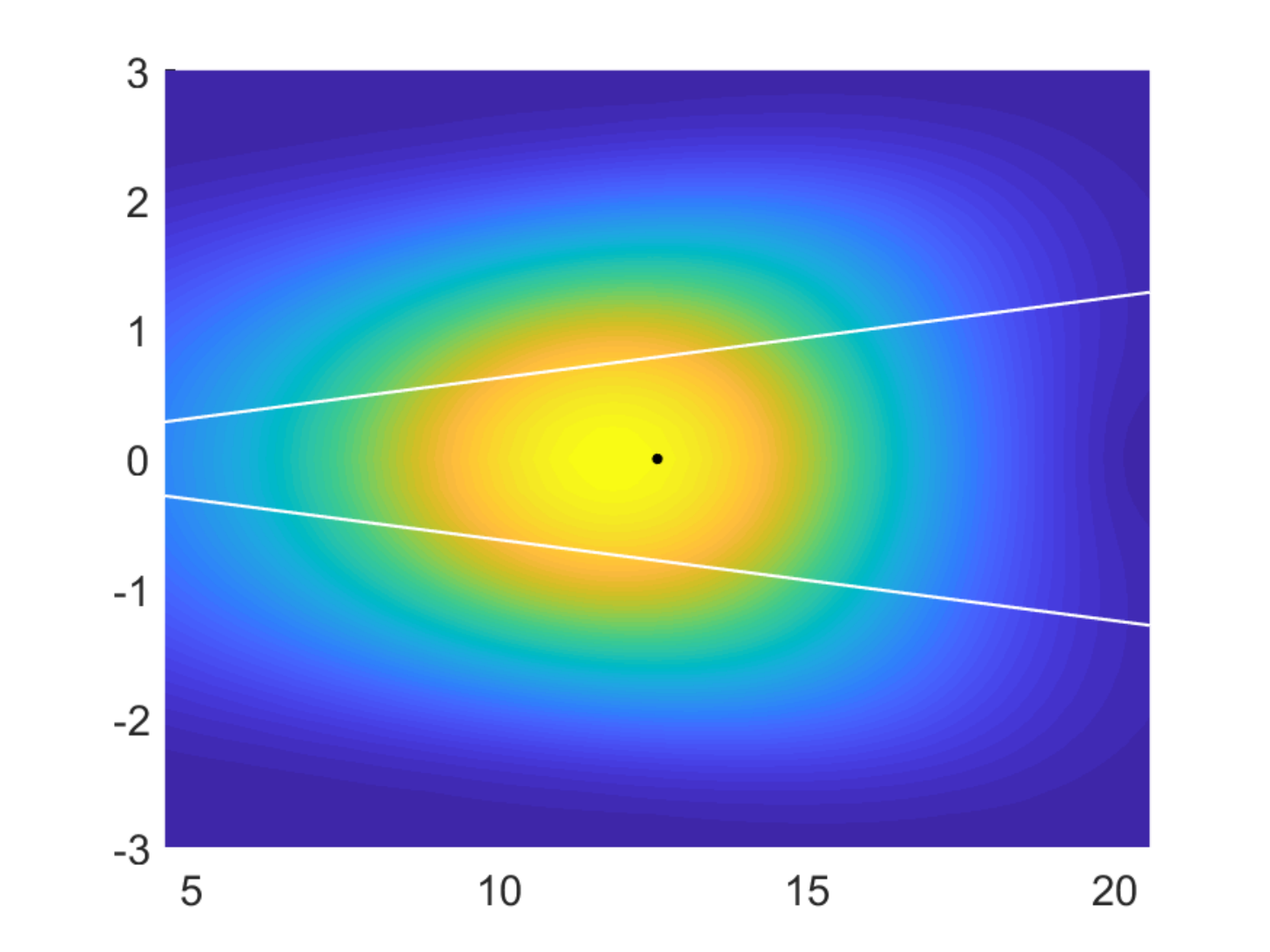}} \\
\subfloat[$\varkappa_1=0.5377$ \ ($\ell=7$)]{\includegraphics[width=4.5cm]{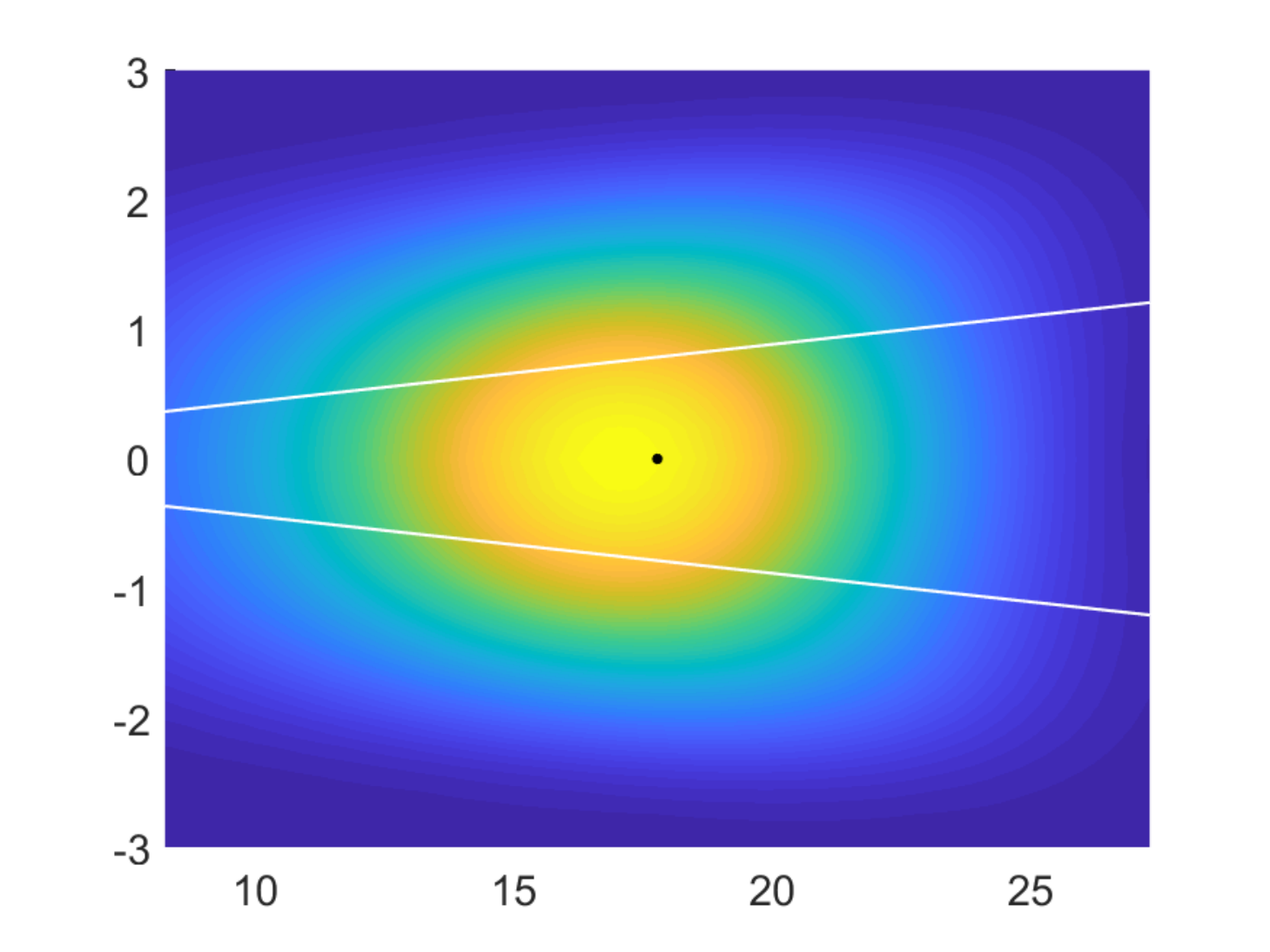}} & 
\subfloat[$\varkappa_1=0.5260$ \ ($\ell=8$)]{\includegraphics[width=4.5cm]{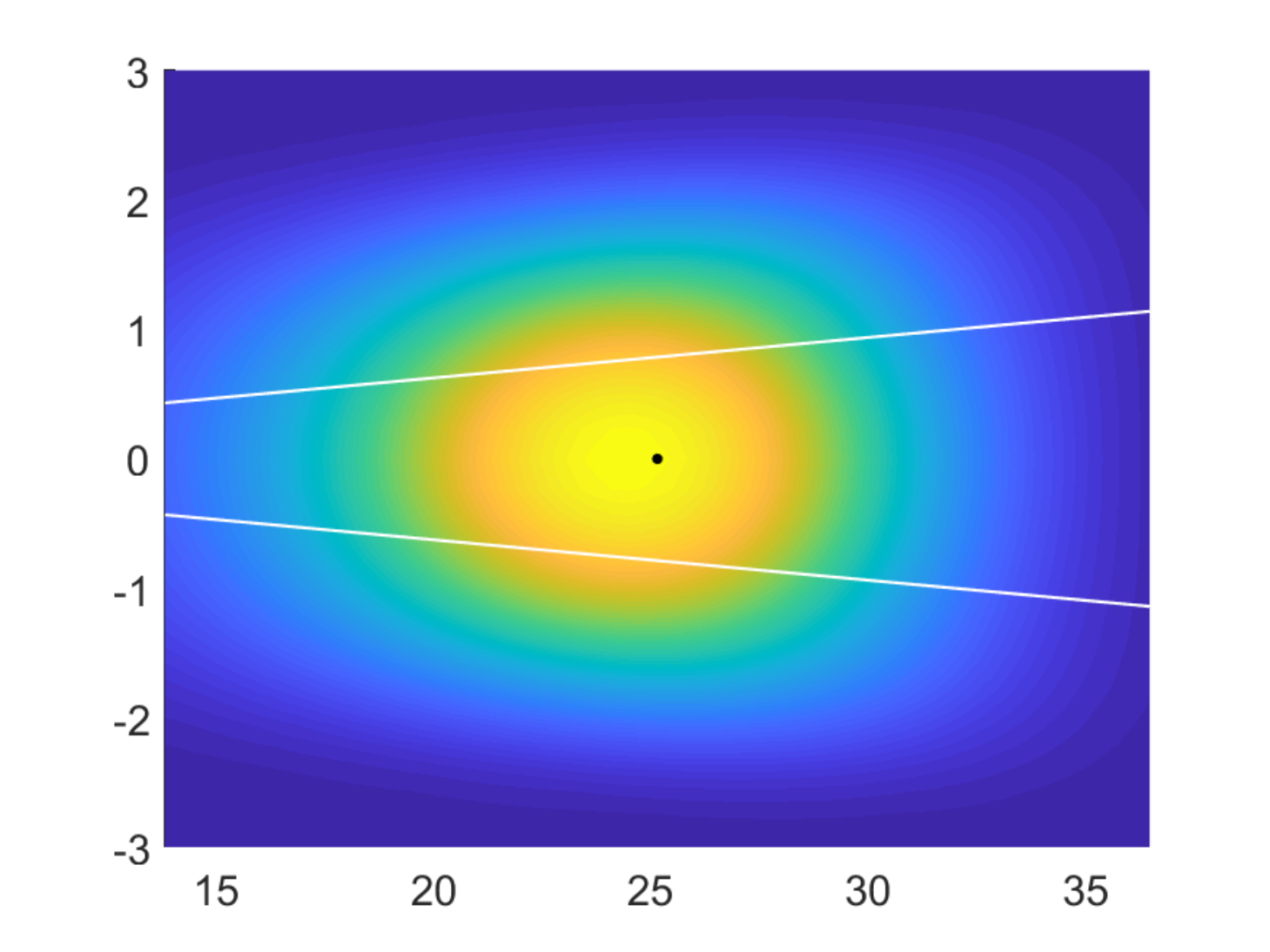}} & 
\subfloat[$\varkappa_1=0.5171$ \ ($\ell=9$)]{\includegraphics[width=4.5cm]{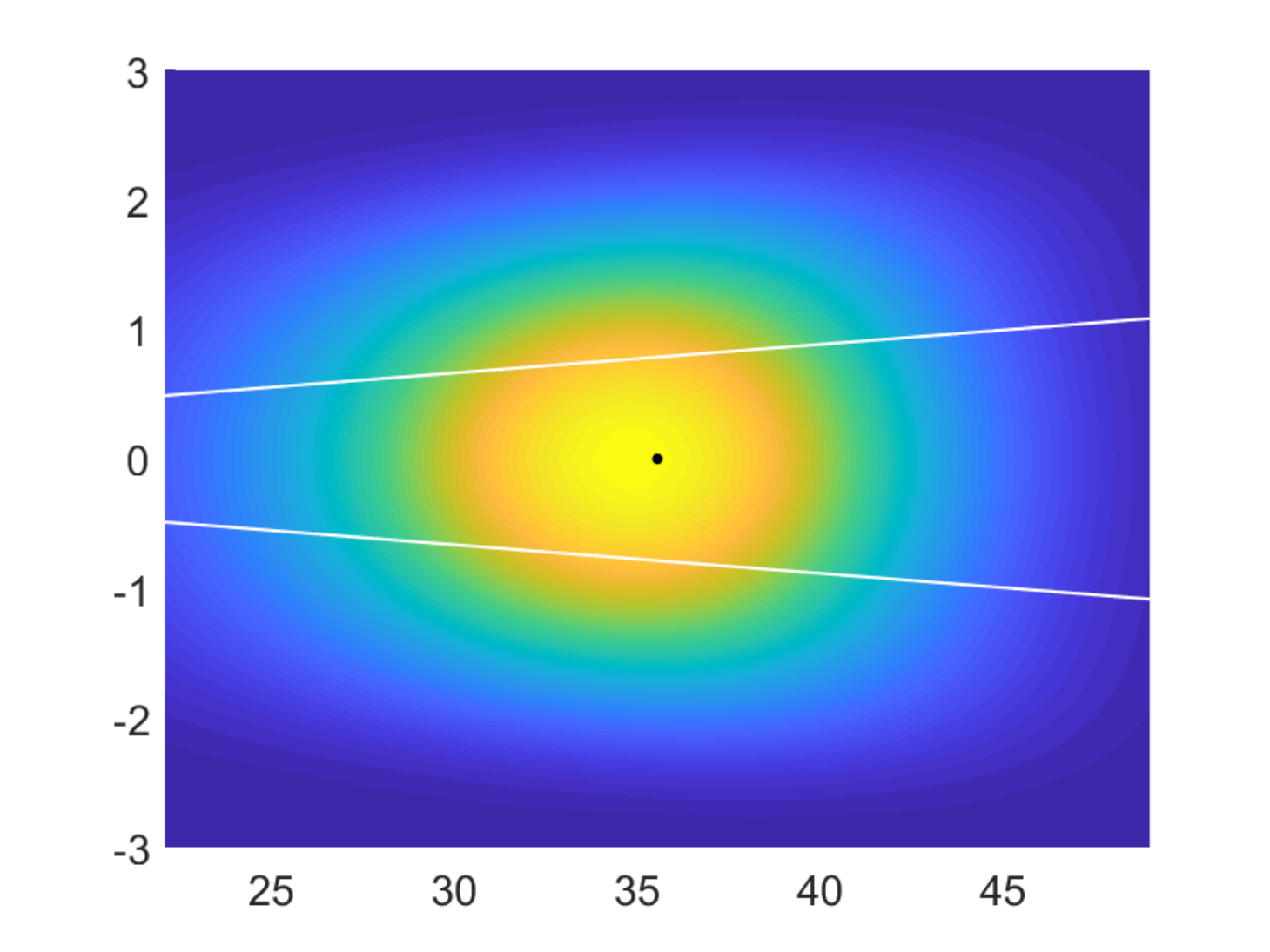}} &
\subfloat[$\varkappa_1=0.5106$ \ ($\ell=10$)]{\includegraphics[width=4.5cm]{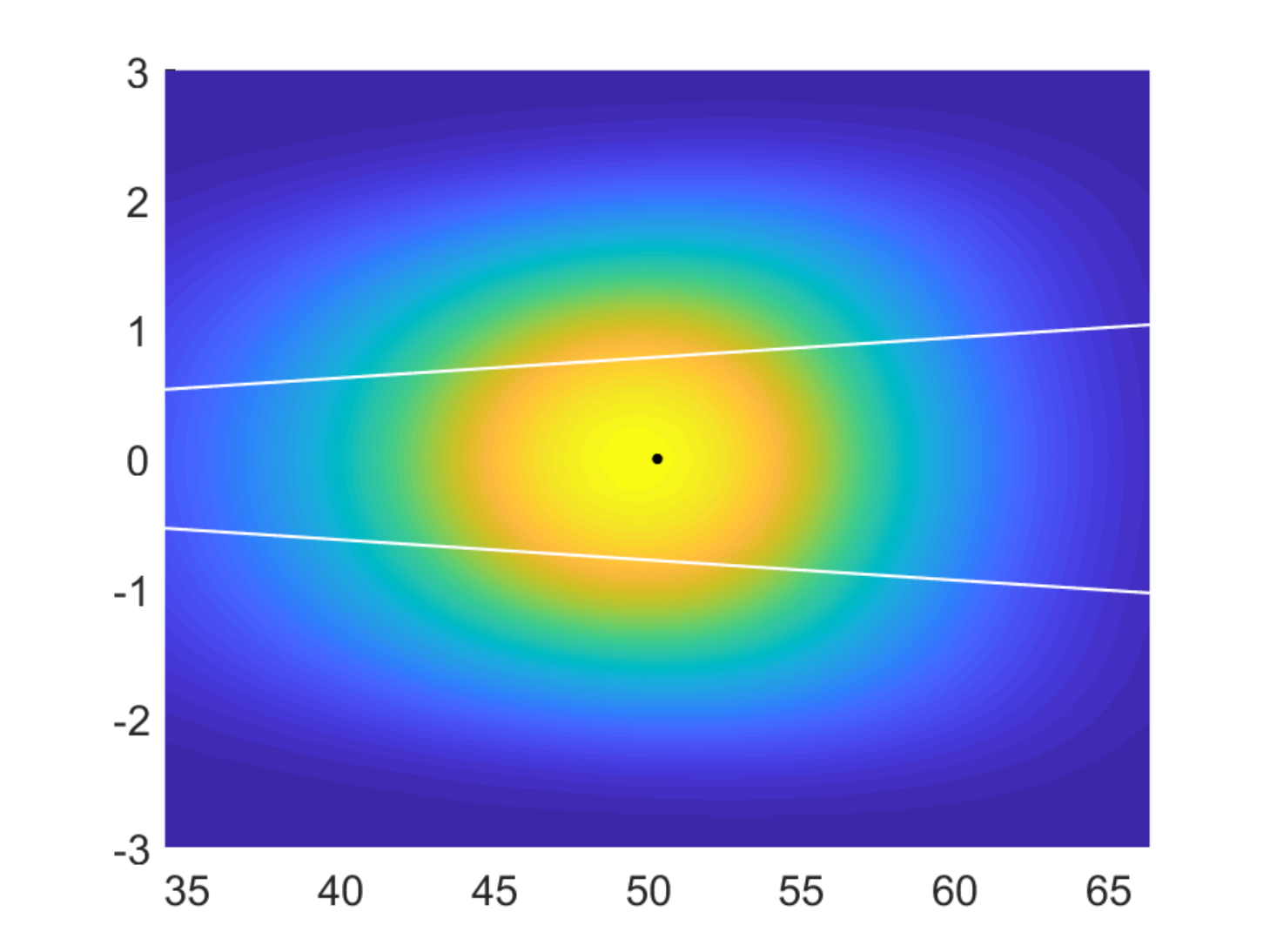}} 
\end{tabular}
\end{center}
\caption{Modulus of the first eigenvector of $\mathcal{X}_{\varepsilon_\ell}$ for $\ell=3,\ldots,10$. The white cross represents the two lines of cancellation of the magnetic field and the black dot marks the point $(\alpha_0/\varepsilon_\ell,0)$.}
\label{F5}
\end{figure}

This choice is driven by the structure of the first term $(\fs,\ft)\mapsto \psi_0(\fs,\ft) = f_0(\fs)u_0(\ft)$ \eqref{DSFchoix2croix} of a possible eigenvector asymptotics for the operator $\fL_\varepsilon$ \eqref{eq:fL}. We note that with $\xi_0=0$, we do not need any gauge transform to go from $\cL_\varepsilon$ to $\fL_\varepsilon$. Thus, using the change of variables \eqref{chgmtechelle2}, we find that
\[
   \psi_0(\fs,\ft) = f_0\Big(\frac{s-\alpha_0}{\sqrt{\varepsilon}}\Big) \,u_0(t).
\]
Going back to the ``physical variables'' $\sigma,\tau$ with \eqref{chgmtechelle}, we find
\[
   \psi_0(\fs,\ft) = f_0\Big(\frac{\varepsilon\sigma-\alpha_0}{\sqrt{\varepsilon}}\Big) \,u_0(\tau).
\]
Taking $\sigma\in [\frac{\alpha_0}{\varepsilon} - \frac{2}{\sqrt{\varepsilon} } \,,\,
    \frac{\alpha_0}{\varepsilon} + \frac{2}{\sqrt{\varepsilon} }]$, we see that $\frac{\varepsilon\sigma-\alpha_0}{\sqrt{\varepsilon}}$ spans the fixed interval $[-2,2]$. Hence we expect that the zoom can yield the image of a convergence as $\varepsilon\to0$. Indeed, this is exactly what we can detect from Figure \ref{F5}.

Thus, in the small angle limit $\varepsilon\rightarrow 0$, there are two centers of localization that are spread at the scale $\varepsilon^{-1/2}$ and go away at the order $\frac{1}{\varepsilon}$.
The ``area of localization'' of the first eigenfunctions goes to infinity. This fact is quite understandable. Considering that the limit $\varepsilon=0$ is singular (the two lines of cancellation become identical and we recover the Montgomery operator $\mathcal{M}^{[2]}$ whose essential spectrum is non-empty); it is rather natural to observe a ``loss of mass at infinity'', characteristic of Weyl sequences.


\bigskip
\end{document}